\documentclass[reqno]{mcom-l}

\usepackage[hidelinks]{hyperref}
\usepackage{amssymb}
\usepackage{amsmath}
\usepackage{amsthm}
\usepackage{mathtools}
\usepackage{enumitem}
\usepackage{algpseudocode}
\usepackage{xcolor}
\usepackage{float}
\usepackage{standalone} 
\usepackage{hyperref}
\usepackage{tikz}         
\usepackage{graphicx}     
\usepackage{pgfplots}
\pgfplotsset{compat=1.17}
\usetikzlibrary{patterns, pgfplots.fillbetween}

\algtext*{EndWhile}
\algtext*{EndFor}
\algtext*{EndIf}

\renewcommand{\leq}{\leqslant}
\renewcommand{\geq}{\geqslant}

\newcommand{\Z}{\mathbb{Z}}

\DeclareMathOperator{\ord}{ord}

\newlist{inputoutputlist}{itemize}{1}
\setlist[inputoutputlist]{label=--,topsep=0pt,leftmargin=20pt}

\copyrightinfo{2025}{Yiming Gao and Yansong Feng}

\newtheorem{theorem}{Theorem}[section]
\newtheorem{lemma}[theorem]{Lemma}
\newtheorem{proposition}[theorem]{Proposition}
\newtheorem{corollary}[theorem]{Corollary}

\theoremstyle{remark}
\newtheorem{remark}[theorem]{Remark}

\theoremstyle{definition}

\newtheorem{algorithm}[theorem]{Algorithm}
\numberwithin{equation}{section}

\begin{document}

\title[Factor via Rank-3 Lattices]{On Factoring and Power Divisor Problems via Rank-3 Lattices and the Second Vector}


%    author one information
\author[Y. Gao]{Yiming Gao}
\address{School of Cyber Science and Technology, University of Science and Technology of China, Hefei, China}
\email{qw1234567@mail.ustc.edu.cn}

\author[Y. Feng]{Yansong Feng}
\address{State Key Laboratory of Mathematical Sciences, Academy of Mathematics and Systems Science, Chinese Academy of Sciences, Beijing, China \and		School of Mathematical Sciences, University of Chinese Academy of Sciences, Beijing, China}
\email{fengyansong@amss.ac.cn}

\author[H. Hu]{Honggang Hu}
\address{School of Cyber Science and Technology, University of Science and Technology of China, Hefei, China \and Hefei National Laboratory, Hefei, China }
\email{hghu2005@ustc.edu.cn}

\author[Y. Pan]{Yanbin Pan}
\address{State Key Laboratory of Mathematical Sciences, Academy of Mathematics and Systems Science, Chinese Academy of Sciences, Beijing, China \and		School of Mathematical Sciences, University of Chinese Academy of Sciences, Beijing, China}
\email{panyanbin@amss.ac.cn}

\begin{abstract}
We propose a deterministic algorithm based on Coppersmith's method that employs a rank-3 lattice to address factoring-related problems. An interesting aspect of our approach is that we utilize the second vector in the LLL-reduced basis to avoid trivial collisions in the Baby-step Giant-step method, rather than the shortest vector as is commonly used in the literature. Our results are as follows:  

- Compared to the result by Harvey and Hittmeir (Math. Comp. 91 (2022), 1367–1379), who achieved a complexity of  
$
O\left( \frac{N^{1/5} \log^{16/5} N}{(\log \log N)^{3/5}} \right)
$  
for factoring a semiprime $N = pq$, we demonstrate that in the balanced $p$ and $q$ case, the complexity can be improved to  
$
O\left( \frac{N^{1/5} \log^{13/5} N}{(\log\log N)^{3/5}} \right).
$

- For factoring sums and differences of powers, i.e., numbers of the form $N = a^n \pm b^n$, we improve Hittmeir's result (Math. Comp. 86 (2017), 2947–2954) from  
$
O(N^{1/4} \log^{3/2} N)
$  
to  
$
O\left( {N^{1/5} \log^{13/5} N} \right).
$

- For the problem of finding $r$-power divisors, i.e., finding all integers $p$ such that $p^r \mid N$, Harvey and Hittmeir (Proceedings of ANTS XV, Res. Number Theory 8 (2022), no.4, Paper No. 94) recently directly applied Coppersmith's method and achieved a complexity of $O\left(\frac{N^{1/4r} \log^{10+\epsilon} N}{r^3}\right)$. By using faster LLL-type algorithm and sieving on small primes, we improve their result to $O\left(\frac{N^{1/4r} \log^{7+3\epsilon} N}{(\log\log N-\log 4r)r^{2+\epsilon}}\right)$. The worst case running time for their algorithm occurs when $N = p^r q$ with $q = \Theta(N^{1/2})$. By focusing on this case and employing our rank-3 lattice approach, we achieve a complexity of  
$
O\left(r^{1/4} N^{1/4r} \log^{5/2} N \right).
$
In conclusion, we offer a new perspective on these problems, which we hope will provide further insights.

\end{abstract}

\maketitle

\section{Introduction}
\label{sec:introduction}

This paper focuses on deterministic and rigorous integer factorization algorithms implemented on a Turing machine. It is worth noting that when these constraints are relaxed, superior complexity bounds have been achieved through various methods, including the Elliptic Curve Method (ECM) \cite{lenstra1987factoring}, the General Number Field Sieve~\cite[\S6.2]{crandall2005prime}, and Shor's Quantum factoring algorithm \cite{shor1994algorithms}.

In \cite{harvey2021exponent}, Harvey demonstrated that a positive integer $N$ can be rigorously and deterministically factored into primes in time $F(N) = O\!\left(N^{1/5} \log^{16/5} N\right)$, where the time complexity is measured in bit operation, that is, the number of steps executed by a deterministic Turing machine equipped with a fixed and finite number of linear tapes. Subsequently, Harvey and Hittmeir \cite{harvey2022log} improved this bound to $F(N) = O\left( \frac{N^{1/5} \log^{16/5} N}{(\log\log N)^{3/5}}\right)$. Our work adopts the same computational model as \cite{harvey2021exponent} and \cite{harvey2022log}.

We briefly introduce the ideas underpinning the algorithms presented in \cite{harvey2021exponent} and \cite{harvey2022log}.
The strategy in \cite{harvey2021exponent} leverages Fermat's Little Theorem. For an integer $\alpha$ coprime to $N=pq$, and integers $a, b$, we have $\alpha^{aq+bp} \equiv \alpha^{aN+b} \pmod p$. If $a/b$ is chosen as a good rational approximation to the unknown ratio $p/q$, then the value $X = aq+bp$ is close to $K = \lfloor(4abN)^{1/2}\rfloor$. This leads to the congruence $\alpha^{X-K} \equiv \alpha^{aN+b-K} \pmod p$. Since $X-K$ is small under the approximation assumption, the core idea is to find a collision modulo $p$ (and hence potentially modulo $N$) between ``baby steps'' of the form $\alpha^i \pmod N$ for small integers $i$, and ``giant steps'' of the form $\alpha^{aN+b-K} \pmod N$ as $a/b$ varies over a suitable dense set of rational approximations. This collision-finding problem is efficiently addressed using techniques from fast polynomial arithmetic, resulting in the complexity bound $O(N^{1/5} \log^{16/5} N)$.

The subsequent work \cite{harvey2022log} refines this approach by incorporating the crucial observation that the prime factors $p$ and $q$ of a large integer $N$ cannot themselves be divisible by small primes. The algorithm is modified to restrict the search space by considering only candidates for $p$ that are coprime to $m = p_1 p_2 \cdots p_d$, the product of the first $d$ primes, for a suitably chosen $d$ such that $m=N^{O(1)}$. The number of residue classes modulo $m$ that need consideration is reduced by a factor of $m/\varphi(m)$, which, by Mertens' theorem, behaves asymptotically like $\log d \approx \log \log m$. When optimized, this yields a saving proportional to $\log \log N$. Technically, this refinement involves a reorganization where suitable pairs $(a,b)$ for the giant steps are no longer generated by simple iteration over a range, but are instead computed using algorithms for finding short vectors in appropriately constructed lattices. This leads to the improved complexity bound $O\left( \frac{N^{1/5} \log^{16/5} N}{(\log\log N)^{3/5}}\right)$.

In this work, we demonstrate that in certain common scenarios, such as when $N = pq$ is a semiprime with $p, q = \Theta(N^{1/2})$, this complexity can be further refined using the Coppersmith method in conjunction with the Baby-step Giant-step framework. Our primary result is as follows:

\begin{theorem}[Deterministic integer factorization]\label{thm:main}
Let $N = pq$ be a semiprime with $p, q = \Theta(N^{1/2})$. Then there exists 
a deterministic algorithm to recover the factors $p$ and $q$ in time
$$
F(N)=O\left( \frac{N^{1/5} \log^{13/5} N}{(\log \log N)^{3/5}} \right).
$$
\end{theorem}

In earlier work, Hittmeir \cite{hittmeir2017deterministic} presented an algorithm for factoring sums and differences of powers with time complexity $F(N)=O(N^{1/4} \log^{3/2} N)$. While this complexity can be improved to $F(N) = O\left( \frac{N^{1/5} \log^{16/5} N}{(\log\log N)^{3/5}}\right)$ using \cite{harvey2022log}, such an approach does not exploit the specific properties of sums and differences of powers. Our framework naturally leverages these properties, leading to our second main result:

\begin{theorem}[Factoring sums and differences of powers]\label{thm:anbn}
Let $a,b\in \mathbb{N}$ be fixed and coprime such that $a>b$, and define $P_{a,b}:=\{a^n\pm b^n: n\in\mathbb{N}\}$. Then, one may compute the prime factorization of any $N\in P_{a,b}$ in time
\[
F(N)=O\left(N^{1/5}\log^{13/5} N\right).
\]
\end{theorem}

We also propose some improved toolkits. We design a proper module choosing algorithm for the loglog speed-up. Furthermore, we present several improvements to existing toolkits in factorization algorithms. While previous works by \cite{markus2018babystep,harvey2021exponent,harvey2022log} required finding an element $\alpha$ with order $\mathrm{ord}_N(\alpha) > N^{2/5}$ or $\mathrm{ord}_N(\alpha) > N^{2/(3+2r)}$, we demonstrate how to reduce these requirements to $\mathrm{ord}_N(\alpha) > N^{1/4+o(1)}$ and $\mathrm{ord}_N(\alpha) > N^{1/(4r)+o(1)}$ respectively. This advancement is achieved through our refined generalization of Harvey's deterministic factorization method for finding $r$-power divisors.

Finally, we extend our rank-3 lattice construction to address the $r$-power divisor problem of finding all integers $p$ such that $p^r \mid N$. While recent work by Hales and Hiary~\cite{hales2024generalization} achieved complexities of $O(N^{1/(r+2)} (\log N)^2 \log \log N)$ and $O(N^{1/(3+2r)} (\log N)^{16/5})$ by extending Lehman's method, and Harvey and Hittmeir~\cite{harvey2022deterministic} obtained $O(N^{1/4r} \log^{10+\epsilon} N/r^3)$ using Coppersmith's method, we present two improvements. First, by incorporating faster LLL-type lattice reduction algorithms and small prime sieving, we improve the complexity to $O\left(\frac{N^{1/4r} \log^{7+3\epsilon} N}{(\log\log N-\log 4r)r^{2+\epsilon}}\right)$. Then, focusing on the worst-case scenario where $N = p^r q$ with $p=\Theta(N^{1/2r})$ and $q =\Theta(N^{1/2})$, our rank-3 lattice construction further reduces the complexity to $O\left({r^{1/4}} N^{1/4r} \log^{5/2} N \right)$.
{

\subsection{Technical Overview}
Our work builds upon the Baby-step–Giant-step (BSGS) framework and the small-prime sieving techniques of \cite{harvey2021exponent,harvey2022log}. We retain the overall strategy of seeking collisions and incorporate the standard $\log\log$ optimization by restricting candidates modulo a product $m$ of small primes (after checking $\gcd(m,N)$, which already factors $N$ if nontrivial).

This subsection summarizes the common template behind our algorithms, explains what we call baby steps, giant steps, and the collisions we search for, and highlights how our giant steps differ from \cite{harvey2021exponent,harvey2022log} via a Coppersmith-style, rank-3 lattice and the use of the second LLL vector.

\medskip
\noindent\textbf{Setup.}
Let $N=pq$ with unknown prime factors $p,q$. Fix $\alpha\in(\mathbb{Z}/N\mathbb{Z})^\times$ of sufficiently large order, choose an integer $m$ (coprime to $N$, otherwise $\gcd(m,N)$ factors $N$) and write
\[
p \;=\; m\,p_{\mathrm{msb}} + p_{\mathrm{lsb}},\qquad 0 \le p_{\mathrm{lsb}} < m,
\]
so that $p_{\mathrm{msb}}$ (resp.\ $p_{\mathrm{lsb}}$) is the most-significant (resp.\ least-significant) part of $p$ in base $m$. Let $m^{-1}$ denote the inverse of $m$ modulo $N$. We set a baby-step budget $k$ and a set of giant-step indices $\mathcal{J}\subseteq (\mathbb{Z}/m\mathbb{Z})^\times$. Let $X$ be an a prior upper bound on $|p_{\mathrm{msb}}|$, typically $X\simeq \lceil N^{1/2}/m\rceil$ in the balanced case.

\medskip
\noindent\textbf{Baby steps.}
Precompute and store
\[
B_i \;=\; \alpha^{\,i m^2}\!\!\!\pmod{N},\qquad 0\le i<k,
\]
together with their indices $i$.

\medskip
\noindent\textbf{Giant steps via a rank-3 lattice (second LLL vector).}
For each $j\in\mathcal{J}$ (the candidate for $p_{\mathrm{lsb}}\equiv p\pmod m$), define
\[
f(x)\;=\;x + j\,m^{-1}\pmod{N}.
\]
If $j=p_{\mathrm{lsb}}$, then $x_0=p_{\mathrm{msb}}$ is a small root of $f(x)\equiv 0\pmod p$ since $m\,x_0\equiv p-j\pmod p$. We adapt Coppersmith’s method by constructing a rank-3 lattice $L$ spanned by the coefficient vectors of
\[
N,\quad f(xX),\quad f(xX)^2,
\]
and apply LLL to obtain a reduced basis. Crucially, we take the \emph{second} LLL vector
\[
v_j \;=\; (c_j,\; b_j X,\; a_j X^2)\in L,
\]
rather than the shortest one. The shortest vector is explicitly $(j,\,mX,\,0)$ (see Lemma~4.4 for details), which enforces $a_j=0$ and leads to a trivial collision. By contrast, the second vector satisfies $a_j\neq 0$, which is essential for the informative quadratic relation below.

Associate to $v_j$ the polynomial $g_j(x)=c_j+b_j x+a_j x^2$. Since $v_j$ is a $\mathbb{Z}$-combination of $N,f,f^2$, we have $p\mid g_j(x_0)$. Writing
\[
g_j(x_0)= i\,p\quad\text{for some integer }i\ge 0,
\]
and using LLL bounds on $\|v_j\|$, one shows that $i$ is small enough so that $i<k$. The corresponding giant step is defined by
\[
y_j \;:=\; \alpha^{\,E_j}\!\!\!\pmod{N},\qquad
E_j \;=\; c_j m^2 + b_j m(1-j) + a_j(1-j)^2,
\]
which is arranged to match the baby-step exponent modulo $p-1$: since $m\,x_0 = p-j\equiv 1-j\pmod{p-1}$, Fermat’s little theorem gives
\begin{align*}
\alpha^{\,i m^2}
&\equiv \alpha^{\,m^2\,g_j(x_0)} \equiv \alpha^{\,m^2(c_j + b_j x_0 + a_j x_0^{2})}
\\
&\equiv \alpha^{\,c_j m^2 + b_j m (1-j) + a_j (1-j)^2}
\;=\; y_j \pmod{p}.
\end{align*}

\medskip
\noindent\textbf{Collision search.}
We seek a collision between the lists $\{B_i\}$ and $\{y_j\}$:
\[
\alpha^{\,i m^2} \;\equiv\; y_j \pmod{N}\quad\text{or at least}\quad \alpha^{\,i m^2} \;\equiv\; y_j \pmod{p},
\]
for some $0\le i<k$ and $j\in\mathcal{J}$. We use the fast product-tree and Bluestein-type multipoint techniques from \cite{harvey2021exponent,harvey2022log} to find all matches in near-linear time in the list sizes. If a collision holds modulo $N$, the identity $g_j(x_0)=ip$ with $a_j\neq 0$ yields the explicit quadratic equation
\[
c_j \;+\; b_j\,\frac{p-j}{m} \;+\; a_j\,\frac{(p-j)^2}{m^2} \;=\; i\,p,
\]
from which we recover $p$. If the collision holds only modulo $p$, batched gcd of evaluations recovers $p$ deterministically.

\medskip
\noindent\textbf{Why the second LLL vector?}
In our 3-dimensional lattice, the shortest vector is $(j,\,mX,\,0)$, which forces a purely linear relation ($a_j=0$) and leads to a trivial collision that carries no information about $p$. Taking instead the second LLL vector guarantees $a_j\neq 0$ and hence a genuinely informative quadratic in the unknown $p_{\mathrm{msb}}$, which is what ultimately drives factor recovery.

\medskip
\noindent\textbf{Parameter choice and complexity.}
Let $\mathcal{J}=(\mathbb{Z}/m\mathbb{Z})^\times$ after sieving small primes as in \cite{harvey2022log}. The dominant cost comes from sorting and matching the baby-step and giant-step lists, of sizes $k$ and $|\mathcal{J}|=\varphi(m)$, respectively, with negligible overhead from $3$-dimensional LLL. Balancing $k\simeq N^{1/2}/m^{3/2}$ and optimizing $m\simeq N^{1/5}$ yields the bound $O(N^{1/5+o(1)})$. For more detailed analysis, see Section~4.
}

\section{Preliminaries}
\label{sec:preliminaries}
\subsection{Notations}

Throughout this paper, we use $\log N$ to denote the binary logarithm $\log_2 N$. For asymptotic complexity analysis, we employ the standard Bachmann-Landau notation: For functions $f, g: \mathbb{N} \to \mathbb{R}^+$, we write $f = O(g)$ if there exist positive constants $c$ and $n_0$ such that $f(n) \leq cg(n)$ for all $n \geq n_0$; we write $f = o(g)$ if $\lim_{n \to \infty} f(n)/g(n) = 0$; and we write $f = \Theta(g)$ if both $f = O(g)$ and $g = O(f)$ hold, or equivalently, if there exist positive constants $c_1$, $c_2$, and $n_0$ such that $c_1g(n) \leq f(n) \leq c_2g(n)$ for all $n \geq n_0$.

\subsection{Arithmetics}

\subsubsection{Integer Arithmetic}
We recall some result about integer and modular arithmetic.

Let $n \in \mathbb{N}^+$, and suppose we are given integers $x, y \in \mathbb{Z}$ satisfying $|x|, |y| \leq 2^n$. The fundamental arithmetic operations exhibit the following computational complexities: The computations of $x + y$ and $x - y$ can be performed in $O(n)$ time. Let $\mathsf{M}(n)$ denote the time complexity of computing the product $xy$; as established in \cite{harvey2021integer}, we have $\mathsf{M}(n) = O(n \log n)$. For $y > 0$, both floor division $\lfloor x/y \rfloor$ and ceiling division $\lceil x/y \rceil$ can be computed in $O(\mathsf{M}(n))$ time, and consequently, the computation of $x \bmod y \in [0, y)$ also requires $O(\mathsf{M}(n))$ time. More generally, for any fixed rational number $u/v \in \mathbb{Q}^+$ and positive integers $x, y > 0$, both $\lfloor (x/y)^{u/v} \rfloor$ and $\lceil (x/y)^{u/v} \rceil$ can be computed in $O(\mathsf{M}(n))$ time. Using the half-GCD algorithm, for any $x, y \in \mathbb{Z}$, we can compute both $g = \gcd(x, y)$ and the Bézout coefficients $u, v \in \mathbb{Z}$ satisfying $ux + vy = g$ in time $O(\mathsf{M}(n) \log n) = O(n \log^2 n)$. In particular, when $\gcd(x, y) = 1$, the modular multiplicative inverse of $x$ modulo $y$ can be computed in $O(\mathsf{M}(n) \log n)$ time.

For any integer $N \geq 2$, we consider the ring of integers modulo $N$, denoted by $\mathbb{Z}_N = \mathbb{Z}/N\mathbb{Z}$. Elements of $\mathbb{Z}_N$ are canonically represented by their residues in $[0, N)$, requiring at most $\lceil \log_2 N \rceil$ bits for storage. Let $\mathbb{Z}_N^*$ denote the multiplicative group of units in $\mathbb{Z}_N$, that is, $\mathbb{Z}_N^* = \{x \in \mathbb{Z}_N : \gcd(x, N) = 1\}$. For the ring $\mathbb{Z}_N$, we have the following computational results: Given $x, y \in \mathbb{Z}_N$, modular addition and subtraction $x \pm y \bmod N$ can be computed in $O(\log N)$ time, while modular multiplication $xy \bmod N$ requires $O(\mathsf{M}(\log N)) = O(\log N \log \log N)$ time. For $x \in \mathbb{Z}_N$ and $m \in \mathbb{N}_0$, the computation of $x^m \bmod N$ can be achieved in $O(\mathsf{M}(\log N) \log m)$ time using the repeated squaring algorithm. Furthermore, testing whether $x \in \mathbb{Z}_N^*$ can be performed in $O(\mathsf{M}(\log N) \log \log N)$ time by computing $\gcd(x, N)$.

\subsubsection{Polynomial Arithmetic}

The next few results are taken from the previous papers
\cite{markus2018babystep},\cite{harvey2021exponent} and \cite{harvey2022log}.

\begin{lemma}[Polynomial Construction]\label{lem:tree}
Let $n \geq 1$ with $n = O(N)$. Given as input $v_1, \ldots, v_n \in \mathbb{Z}_N$, we may compute the polynomial $f(x) = (x - v_1) \cdots (x - v_n) \in \mathbb{Z}_N[x]$ in time
$O(n \, {\log}^3 N).$
\end{lemma}

\begin{lemma}[Polynomial Evaluation]\label{lem:bluestein}
Given as input an element $\alpha \in \mathbb{Z}_N^*$, positive integers $m, n = O(N)$, and $f \in \mathbb{Z}_N[x]$ of degree $n$, we may compute $f(1), f(\alpha), \ldots, f(\alpha^{m-1}) \in \mathbb{Z}_N$ in time
$$O((n + m) \,{\log}^2 N).$$
\end{lemma}

\begin{proof}
This is exactly in \cite{harvey2021exponent,harvey2022log}. The proof leverages Bluestein’s trick to compute polynomial evaluations efficiently by transforming the problem into a Laurent polynomial multiplication.
\end{proof}

\subsection{Lattice}
We begin by recalling fundamental concepts of lattices and basis reduction. Although our work exclusively employs integer lattices, all definitions extend naturally to real-valued lattices.

Let $v_1, \ldots, v_n \in \mathbb{Z}^m$ with $m \geq n$ be linearly independent vectors. The lattice $\mathcal{L}$ generated by $\{v_1, \ldots, v_n\}$ is defined as
$$
\mathcal{L} = \left\{ \sum_{i=1}^n a_i v_i \;\big|\; a_i \in \mathbb{Z} \right\} \subseteq \mathbb{Z}^m.
$$
When $m = n$, the lattice is said to be of full rank. A basis $\mathbf{B} = \{v_1, \ldots, v_n\}$ spans $\mathcal{L}$, and we call $n = \dim(\mathcal{L})$ the lattice dimension.

Given a basis $\mathbf{B}$, let $v_1^*, \ldots, v_n^*$ denote its Gram-Schmidt orthogonalization. The lattice determinant is
$$
\det(\mathcal{L}) := \prod_{i=1}^n \|v_i^*\|,
$$
where $\| \cdot \|$ denotes Euclidean norm. Any lattice $\mathcal{L}$ admits infinitely many bases, yet all share the same determinant. For full-rank lattices, this equals the absolute value of the determinant of the basis matrix $[v_1 \cdots v_n]^\top$.

Let $\mathcal{B}_m(0, r) := \{ \mathbf{x} \in \mathbb{R}^m : \|\mathbf{x}\| < r \}$ denote the $m$-dimensional open ball of radius $r$ centered at the origin (subscript omitted when clear from context). The successive minima $\lambda_1, \ldots, \lambda_n$ of a rank $n$ lattice satisfy that $\lambda_i(\mathcal{L})$ is the smallest radius of a ball centered at the origin containing $i$ linearly independent lattice vectors.

Minkowski's Second theorem establishes a fundamental connection between a lattice's determinant and its successive minima.

\begin{theorem}[Minkowski's Second Theorem]\label{thm:minkowski}
For any rank $n$ lattice $\mathcal{L}$ with basis $\mathbf{B}$, the successive minima satisfy
$$
\prod_{i=1}^n \lambda_i(\mathcal{L}) < (\sqrt{n})^n \det(\mathbf{B}).
$$
\end{theorem}

This bound is non-constructive. For the special case of dimension 2, the classical Gauss reduction algorithm efficiently computes a shortest lattice vector. In arbitrary dimensions, we employ the seminal lattice reduction technique by Lenstra, Lenstra, and Lovász \cite{Lenstra1982factoring}:

\begin{lemma}[LLL Reduction]\label{lem:lll}
Let $\mathcal{L}(\mathbf{B})\subseteq\mathbb{Z}^m$ be a lattice with basis $\{b_1, \ldots, b_n\}$. The LLL algorithm outputs a reduced basis $\{v_1, \ldots, v_n\}$ satisfying
$$
\|v_i\| \leq 2^{(n-1)/2} \lambda_i(\mathcal{L}) \quad (1 \leq i \leq n),
$$
with time complexity $O(n^{4}m\beta^{2}(\log n+\log \beta))$, where $\beta = \Theta(\log \|\mathbf{B}\|_{\max}+\log m)$ encodes the basis vector bit-length. 
\end{lemma}
\begin{remark}
By Corollary 17.5.4 in Mathematics of Public Key Cryptography, Version 2.0 by Steven D. Galbraith (October 31, 2018)~\cite{galbraith2012mathematics}, we know that the LLL algorithm requires $O(n^3 m \beta)$ arithmetic operations on integers of size $O(n \beta)$. Here, by employing integer multiplication with complexity $O(n \log n)$ for two $n$-bit integers as described in~\cite{harvey2021integer}, we obtain our LLL complexity result.
\end{remark}
\begin{remark}
\label{rem:fasterlll}
If we only require the first relative shortest vector, we can apply the lattice reduction algorithm proposed by Neumaier and Stehlé~\cite{NS2016fasterlll}, which improves the time complexity to $O(n^{4+\epsilon} \beta^{1+\epsilon})$.
\end{remark}

\subsection{Coppersmith's Method} In Coppersmith's method, the first step is to construct a lattice, where each row corresponds to the coefficient vector of a polynomial sharing the same root modulo $M$. To find small roots of a modular univariate polynomial $f$, we require a polynomial that shares the same root with $f$ not only modulo $M$ but also over $\mathbb{Z}$. To achieve this, we use Howgrave-Graham's following result~\cite{howgrave1997finding,howgrave2001approximate}, which establishes that small modular roots of a polynomial $h$ with small coefficients are indeed integer roots of $h$.

\begin{lemma}[Howgrave-Graham~\cite{howgrave1997finding}]
\label{lem:HG}
Let $h(x_1, \ldots, x_n) \in \mathbb{Z}[x_1, \ldots, x_n]$ be the sum of at most $w$ monomials, and let $X_1, \ldots, X_n > 0$. For any $(y_1, \ldots, y_n) \in \mathbb{Z}^n$ satisfying the following two conditions:
\begin{enumerate}
    \item $h(y_1, \ldots, y_n) \equiv 0 \pmod{M}$ where $|y_i| < X_i$ for $1 \leq i \leq n$,
    \item $\Vert h(x_1 X_1, \ldots, x_n X_n) \Vert < \frac{1}{\sqrt{w}} M$,
\end{enumerate}
then $h(y_1, \ldots, y_n) = 0$.
\end{lemma}

In 1996, Coppersmith proposed a result called \textit{factoring with a hint}~\cite{coppersmith1996finding}, which factors RSA moduli $N = pq$ in polynomial time given only half of the bits of $p$. {This result was further generalized to more general cases.} 

\begin{lemma}[\cite{may2003new}, Theorem 7]
\label{lem:fwh}
Let $N$ be an integer of unknown factorization with a divisor $b \geq N^{\beta}$. Let $f_\delta(x)$ be a univariate, monic polynomial of degree $\delta$. Then all solutions $x_0$ to the equation  
\[
f_\delta(x) \equiv 0 \pmod{b} \quad\text{with}\quad |x_0| \leq cN^{\beta^2 / \delta}
\]  
can be found in time $O\left(\lceil c\rceil\frac{\log^{6+3\epsilon} N}{\delta^{1+\epsilon}}\right)$.
\end{lemma}

\begin{proof}
We begin by defining
$$
X \coloneqq \frac{1}{2}N^{\frac{\beta^2}{\delta}-\frac{1}{\log N}}, \quad
n \coloneqq \left\lceil\log N + 1\right\rceil, \quad
m \coloneqq \left\lceil\frac{\beta n}{\delta}\right\rceil.
$$
We construct a set $G$ of polynomials where each polynomial has a root $x_0$ modulo $b^m$ whenever $f_b(x)$ has the root $x_0$ modulo $b$. The set comprises:

$$
\begin{array}{lllll}
N^m & x N^m & x^2 N^m & \cdots & x^{\delta-1} N^m \\
N^{m-1}f & xN^{m-1}f & x^2N^{m-1}f & \cdots & x^{\delta-1}N^{m-1}f \\
N^{m-2}f^2 & xN^{m-2}f^2 & x^2N^{m-2}f^2 & \cdots & x^{\delta-1}N^{m-2}f^2 \\
\vdots & \vdots & \vdots & & \vdots \\
Nf^{m-1} & xNf^{m-1} & x^2Nf^{m-1} & \cdots & x^{\delta-1}Nf^{m-1}.
\end{array}
$$
Additionally, we include:
$$
f^m, \, xf^m, \, x^2f^m, \, \dots, \, x^{n-\delta m}f^m.
$$
{
Namely, we have chosen the polynomials
$$
\begin{array}{ll}
g_{i, j}(x)=x^j N^i f^{m-i}(x) & \text { for } i=0, \ldots, m-1, j=0, \ldots, \delta-1 \text { and }, \\
h_i(x)=x^i f^m(x) & \text { for } i=0, \ldots, t-1.
\end{array}
$$
}
Let $L$ be the lattice spanned by the coefficient vectors of $g_{i,j}(xX)$ and $h_i(xX)$. The basis matrix $B$ of $L$ is lower triangular, yielding:
$$
\det(L) = N^{\frac{1}{2}\delta m(m+1)}X^{\frac{1}{2}n(n-1)}.
$$
Applying Lemma \ref{lem:HG}, we require:
$$
2^{\frac{n-1}{4}}\det(L)^{\frac{1}{n}} < \frac{b^m}{\sqrt{n}}.
$$
Using $b \geq N^\beta$, this yields:
$$
N^{\frac{\delta m(m+1)}{2n}}X^{\frac{n-1}{2}} \leq 2^{-\frac{n-1}{4}}n^{-\frac{1}{2}}N^{\beta m}.
$$
For $X$, we obtain:
$$
X \leq 2^{-\frac{1}{2}}n^{-\frac{1}{n-1}}N^{\frac{2\beta m}{n-1}-\frac{\delta m(m+1)}{n(n-1)}}.
$$
For $n \geq 7$, $n^{-\frac{1}{n-1}} = 2^{-\frac{\log n}{n-1}} \geq 2^{-\frac{1}{2}}$, simplifying to:
$$
X \leq \frac{1}{2}N^{\frac{2\beta m}{n-1}-\frac{\delta m(m+1)}{n(n-1)}}.
$$
Given our choice of $X$, it suffices to show:
$$
\frac{m(2\beta n-\delta(m+1))}{n(n-1)} \geq \frac{\beta^2}{\delta}-\frac{1}{\log N}
$$
Substituting $m = \beta n/\delta$ and simplifying yields:
$$
\frac{1}{\log N} \geq \frac{\beta(\delta-\beta)}{\delta(n-1)}.
$$
This is equivalent to:
$$
n-1 \geq \frac{\beta(\delta-\beta)}{\delta}\log N.
$$
Since $0 < \beta \leq 1$, we have $\frac{\beta(\delta-\beta)}{\delta}\log N < \log N$, which holds by our choice of $n$.

Note that our choice of $X$ only covers solutions in $[-X, X]$. To find all solutions in $[-cN^{\frac{\beta^2}{\delta}}, cN^{\frac{\beta^2}{\delta}}]$, we solve the problem for at most $4\lceil c\rceil$ translated polynomials to cover the full range. At last we apply a root-finding algorithm on each polynomial. For a single polynomial, the total time complexity is dominated by the LLL algorithm. Note that we only require the first relative shortest vector, so we could apply the faster lattice reduction algorithm in Remark \ref{rem:fasterlll}. With lattice dimension $n = O(\log N)$ and basis vector bit-length $\log(N^m) = m\log N$, the complexity is:
$$
O(\lceil c\rceil\log^{4+\epsilon} N(m\log N)^{1+\epsilon}) = O\left(\lceil c\rceil\frac{\log^{6+3\epsilon} N}{\delta^{1+\epsilon}}\right).
$$
\end{proof}

\subsection{Prime Distribution}
Following the idea of \cite{harvey2022log}, we also need the prime distribution lemma for the loglog speedup.

\begin{lemma}\label{lem:Prime Distribution}
For $B \to \infty$ we have
\begin{align}
\sum_{\substack{2 \leqslant r \leqslant B \\ r \text { prime }}} {\ln} r=(1+o(1)) B,\label{eq:prime1}
\\ \prod_{\substack{2 \leqslant r \leqslant B \\ r \text { prime }}} \frac{r-1}{r}=\Theta\left(\frac{1}{{\ln} B}\right) \label{eq:prime2}.
\end{align}
\end{lemma}

To choose a proper module, we design the following lemma:

\begin{lemma}\label{lem:Prime Product Construction}
For sufficiently large input $x \in \mathbb{N}$, there exists a deterministic polynomial-time algorithm that outputs an integer $m = \prod_{p \in S} p$ where $S$ is a set of distinct primes, satisfying: 
$$
\frac{x}{2} < m < 2x, \quad \frac{\varphi(m)}{m} = \Theta\left( \frac{1}{\log\log x} \right)
$$
\end{lemma}

\begin{proof}
Label the first $n$ primes by $p_1, p_2, \dots, p_n$. We compute the product $\prod_{i=1}^n p_i$ step by step, stopping at the first stage where the product exceeds $x$. Suppose this occurs at the $k$-th prime. Then
$$
\prod_{i=1}^{k-1} p_i < x \le \prod_{i=1}^k p_i.
$$
Define $t := \left\lfloor \prod_{i=1}^k p_i / x \right\rfloor$. By Bertrand's postulate, there exists a prime $p$ with $t < p < 2t$. We scan integers from $t+1$ to $2t-1$ to find $p_s$, the prime closest to $t$. We claim that $p_s \leq p_k$. Indeed,
$$
p_k = \frac{\prod_{i=1}^k p_i}{\prod_{i=1}^{k-1} p_i} > \frac{\prod_{i=1}^k p_i}{x} \geq t.
$$
Define $m := \prod_{i=1}^k p_i / p_s$, we have
$$
\frac{x}{2}\le\frac{\prod_{i=1}^{k}p_i}{2t}<m=\frac{\prod_{i=1}^{k}p_i}{p_s}<\frac{\prod_{i=1}^{k}p_i}{t}<2x.
$$
We now estimate $k$. From \eqref{eq:prime1}, we know this implies $p_k=\Theta(\log x)$. Consequently, by \eqref{eq:prime2}
$$
\frac{\varphi(m)}m=\frac{p_s}{p_s-1} \prod_{\substack{2 \leqslant r \leqslant p_k \\ r \text { prime }}} \frac{r-1}{r}=\Theta\left(\frac{1}{\log p_k}\right) =\Theta\left(\frac{1}{\log\log x}\right).
$$
Since the algorithm only needs to check primes up to $p_k$, the total time complexity is $(\log x)^{O(1)}$, which is a polynomial in $\log x$.
\end{proof}

\section{Some Improved Toolkits}
\label{sec:betterTools}
All related works involve finding an element $\alpha$ of large order. More precisely, the works~\cite{markus2018babystep,harvey2021exponent,harvey2022log} require an $\alpha$ with $\mathrm{ord}_N(\alpha) > N^{2/5}$, while~\cite{hales2024generalization} requires $\mathrm{ord}_N(\alpha) > N^{2/(3+2r)}$. In our work, we improve these requirements to $\mathrm{ord}_N(\alpha) > N^{1/4+o(1)}$ for the former setting, and $\mathrm{ord}_N(\alpha) > N^{1/(4r)+o(1)}$ for the latter. {
In a recent independent work~\cite{oznovich2025deterministically}, the condition is further relaxed to \(N^{1/6r} \le \delta\) by a more refined analysis while the overall algorithmic framework remains unchanged.}

First, we generalize the result of Theorem 1.1 in~\cite{harvey2022deterministic} as the following Theorem.

\begin{theorem}
\label{thm:prqfactorsmall}
Let $N,s$ be a natural number and $m \in \mathbb{Z}_N^*$ such that $s, m < N$. Knowing $s$ and $m$, one can compute finds all divisor $p$ such that $p \equiv s \pmod{m},p^r | N$ in 
$$
O\left(\left\lceil\frac{ N^{1/4r}} {m}\right\rceil \frac{\log^{7+3\epsilon} N}{r^{2+\epsilon}} \right)
$$
bit operations.
\end{theorem}
\begin{proof}
See Appendix~\ref{sec:appendixProofprq}.
\end{proof}
\begin{corollary}
\label{coro:pqmodm}
Let $N,s$ be a natural number and $m \in \mathbb{Z}_N^*$ such that $s, m < N$ and $p \equiv s \pmod{m}$ for every prime divisor $p$ of $N$. Knowing $s$ and $m$, one can factor $N$ in 
$$
O\left(\left\lceil\frac{ N^{1/4}} {m}\right\rceil \log^{7+3\epsilon} N \right)
$$
bit operations.
\end{corollary}
\begin{proof}
Set $r=1$ in Theorem~\ref{thm:prqfactorsmall}.
\end{proof}
\begin{remark}
When $r=1$, Corollary~\ref{coro:pqmodm} improves Theorem 3.1 in \cite{markus2018babystep}, which obtains 
$$
O\left(\frac{ N^{1/4}} {\sqrt{m}} \log^2 N \right)
$$
when $m$ is larger than $\log^{10} N$. Another advantage is that our algorithm only requires polynomial space.
\end{remark}

Then we revisit the Order-finding algorithm (Lemma 4.1) in~\cite{hales2024generalization}.

\begin{lemma}
\label{lemma:bigOrderprq}
There is an algorithm taking as an input $N\geq 2$ and $\delta$ such that $N^{1/4r}\log^8 N\leq \delta\leq N$. It returns either some $\alpha\in\Z_N^*$ with $\ord_N(\alpha)>\delta$,
or a nontrivial factor of $N$, or ``$N$ is $r$ power free''.
Its runtime is bounded by
\[
O\left(\frac{\delta^{1/2}\log^2 N}{(\log\log \delta)^{1/2}} \right).
\]
\end{lemma}
\begin{proof}
See Appendix~\ref{sec:appendixFindprq}.
\end{proof}

Finally, we revisit the Order-finding algorithm in~\cite{markus2018babystep,harvey2021exponent,harvey2022log}.

\begin{lemma}
\label{lemma:improvehit2018}
There is an algorithm with the following properties.
It takes as input integers $N\geq 2$ and $\delta$ such that $N^{1/4}\log^8 N\leq \delta\leq N$.
It returns either some $\alpha\in\Z_N^*$ with $\ord_N(\alpha)>\delta$,
or a nontrivial factor of $N$, or ``$N$ is prime''.
Its runtime is bounded by
\[
O\left(\frac{\delta^{1/2}\log^2 N}{(\log\log \delta)^{1/2}} \right).
\]
\end{lemma}
\begin{proof}
Setting $r=1$ in Lemma~\ref{lemma:bigOrderprq}.
\end{proof}
\begin{remark}
In \cite[Remark~6.4]{markus2018babystep}, the author conjectured that the restriction could potentially be relaxed to $\delta \ge N^{1/3+o(1)}$ for suitable $o(1)$. Through the application of Theorem~\ref{thm:prqfactorsmall}, we successfully improve the lower bound of $\delta$ from $N^{2/5}$ to $N^{1/4+o(1)}$. Given that the algorithm's complexity is $\delta^{1/2+o(1)}$, this lemma remains applicable for potential future improvements in deterministic integer factorization algorithms targeting complexities of $N^{1/6+o(1)}$ or even $N^{1/8+o(1)}$.
\end{remark}

\section{Starting Point: Factoring $N=pq$}
\label{sec:factor}
\subsection{Factoring Algorithm for Balance Semiprime}

For the convenience of the reader,
we recall the following algorithm from \cite{harvey2021exponent},
which forms a key subroutine of the main search algorithm presented afterwards.
\begin{algorithm}[Finding collisions]\ \\
\label{alg:collisions}
\textit{Input:}
\begin{inputoutputlist}
\item A positive semiprime $N=pq$.
\item A positive integer $\kappa$, and an element {$\gamma$}\,$\in\Z_N^*$ such that $\ord_N({\gamma}) \geq \kappa$.
\item Elements $v_1,\ldots,v_n\in\Z_N$ for some positive integer $n$,
such that $v_h\neq {\gamma}^i$ for all $h\in\{1,\ldots,n\}$ and $i\in\{0,\ldots,\kappa-1\}$.
\item There exists $h\in\{1,\ldots,n\}$ such that
$v_h\equiv {\gamma}^i \pmod p$ or $v_h\equiv {\gamma}^i \pmod q$ for some
$i\in\{0,\ldots,\kappa-1\}$.
\end{inputoutputlist}
\textit{Output:}
\begin{inputoutputlist}
\item $p$ and $q$.
\end{inputoutputlist}
\begin{algorithmic}[1]
\State Using Lemma \ref{lem:tree} (product tree), compute the polynomial
\[
f(x) \coloneqq (x-v_1)\cdots (x-v_n)\in\Z_N[x].
\]
\State Using Lemma \ref{lem:bluestein} (Bluestein's algorithm), compute the values $f({\gamma}^i)\in\Z_N$ for $i=0,\ldots,\kappa-1$.
\For{$i=0,\ldots,\kappa-1$}
\State Compute $\gamma_i \coloneqq \gcd(N,f({\gamma}^i))$.
\If {$\gamma_i\notin\{1,N\}$} recover $p$ and $q$ and return.
\EndIf
\If{$\gamma_i=N$}
\For{$h=1,\ldots,n$}
\If {$\gcd(N,v_h-{\gamma}^i)\neq 1$} recover $p$ and $q$ and return.
\EndIf
\EndFor
\EndIf
\EndFor
\end{algorithmic}
\end{algorithm}

\begin{proposition}
\label{prop:findCollisions}
Algorithm \ref{alg:collisions} is correct.
Assuming that $\kappa,n=O(N)$, its running time is
$O(n\log^3 N + \kappa\log^2 N)$.
\end{proposition}
\begin{proof}
This is exactly Proposition 4.1 in \cite{harvey2021exponent} and Proposition 4.2 in \cite{harvey2022log}.
\end{proof}

We now present the main search algorithm, the idea is to generalize Coppersmith's method by relaxing the determinant constraint $\det(L) < p^{md}$. 

\begin{algorithm}[The main search]\ \\   % force line break
\label{alg:search}
\textit{Input:}
\begin{inputoutputlist}
\item A semiprime $N =pq$ with $cN^{\beta}<p \le N^{\beta},1/3\le\beta\le 1/2$, where $c<1$ is a constant.
\item Positive integers $m,{m^{-1} \pmod N}\in \mathbb{Z}_N^*$ such that $72< m < N^{(1-\beta)/2}/2, X=\left\lfloor \frac{ N^{\beta}}{m }\right\rfloor$.

\item An element $\alpha\in\Z_N^*$ such that for all $i\in [k]$, $\gcd(\alpha^{m^2i}-1,N)= 1$ where
$$
k \coloneqq \left\lceil \frac{2\cdot 3^{5/4}N^{1/2}}{cm^{3/2}}\right\rceil.
$$
\end{inputoutputlist}
\textit{Output:}
$p$ and $q$.
\begin{algorithmic}[1]
\For{$i=0,\ldots,k-1$} \label{step:babystep-loop} \Comment{Computation of babysteps}
\State Compute $\alpha^{m^2i} \pmod N$. \label{step:babystep}

\EndFor
\For{$j=1, \ldots, m$} \Comment{Computation of giantsteps}
    \If{$\gcd(j,m) = 1$} \label{step:jmgcd}
        \State \label{step:lattice-construction}
        Construct 3-rank lattice with basis
        \begin{equation}
        \label{eq:lattice}
        B=\left(\begin{array}{ccc}
        N & 0 & 0\\
        j{m^{-1}} & X& 0\\
        j^2{m^{-2}} & 2j{m^{-1}}X & X^2
        \end{array}\right)
        \end{equation}
        \State \label{step:L1}
        Apply Lemma \ref{lem:lll} (LLL Algorithm) and take the second vector
        \label{step:second vector}
        $$
        v_{j}=(c_{j},b_{j}X,a_{j}X^2).
        $$
        \State \label{step:giantstep}
        Compute 
        $$
        x_{j}=\alpha^{c_{j}m^2+b_{j}m(1-j)+a_{j}(1-j)^2} \pmod N
        $$
    \EndIf
\EndFor
\State \label{step:match}%
Applying a sort-and-match algorithm to the babysteps and giantsteps
computed in Steps \ref{step:babystep} and \ref{step:giantstep},
find all pairs $(i, j)$ such that
\begin{equation}
\label{eq:modn}
\alpha^{m^2 i} \equiv x_{j} \pmod N.
\end{equation}
For each such match, solve the quadratic equation
$$
c_{j}+b_{j}\frac{(p-j)}{m}+a_{j}\frac{(p-j)^2}{m^2}=ip
$$
If $p$ is found, return $p$ and $N/p$.

\State \label{step:collisions}%
Let $x_{1},\ldots, x_{n}$ be the list of giantsteps computed in Step \ref{step:giantstep},
skipping those that were discovered in Step \ref{step:match}
to be equal to one of the babysteps.
Apply Algorithm \ref{alg:collisions} (finding collisions)
with $N$, $\kappa \coloneqq k$,
{\(\gamma \coloneqq \alpha^{m^2} \pmod N\)} and $x_{1},\ldots, x_{n}$ as inputs.
\State \label{step:mainreturn}%
Return $p$ and $q$.
\end{algorithmic}
\end{algorithm}

Before we prove the correctness of Algorithm \ref{alg:search}, we would like to explain the core step of the main search.

The main idea is to use the Coppersmith method to compute the giantsteps, which is different with \cite{harvey2022log}. We also gain the loglog speed up by sieving on small primes. In common Coppersmith method, suppose one knows $x,m$ such that $p \equiv x\bmod m$, if $\log_N (p/m)\le  (\log_Np)^2 $ then one could factor $N$ in polynomial time. We use the same lattice in dimension $3$. While the classical Coppersmith method requires strict adherence to the Howgrave-Graham lemma's inequality, we present a generalization that relaxes these constraints. Specifically, we remove the strict requirement that $\det < p^d$. Although the vectors obtained from our modified LLL approach may not be as short, they can be utilized to construct giant steps. We then employ Harvey's baby-step giant-step method~\cite{harvey2021exponent} to find collisions and factor $N$.

The reason why we choose the second vector by LLL algorithm, is that the shortest vector is the second row of the basis itself. The collision by this shortest vector is trivial. We prove this in the following lemma.

\begin{lemma}
\label{lem:second vector}
Let $N\in \mathbb{N},m,j \in \mathbb{Z}_m^{*}$. Let ${m^{-1}}\in\mathbb{Z}_N^*$ denote the multiplicative inverse of $m$ modulo $N$. Define ${t \coloneqq \frac{m\,m^{-1}-1}{N}}\in\mathbb{Z}$, $\tilde{p}_j = j{m^{-1}}$, $X = \lfloor N^{\beta}/m \rfloor$, $1/3 \leq \beta \leq 1/2, 72 < m < N^{(1-\beta)/2}/2$. Then the shortest vector in the lattice

$$
B = \begin{pmatrix}
N & 0 & 0\\
\tilde{p}_j & X & 0\\
\tilde{p}_j^2 & 2\tilde{p}_jX & X^2
\end{pmatrix}
$$
is 
$$
v_0 = -tj(N, 0, 0) + m({\tilde{p}_j}, X, 0) = (j, mX, 0),
$$
and the second vector $v_2$ obtained from the LLL algorithm has a non-zero third coordinate.
\end{lemma}

\begin{proof}
We first demonstrate that $v_0$ is indeed the shortest vector in the lattice. Suppose, for the sake of contradiction, that there exists a shorter vector. Let this vector be represented as
$$
v = (aN+bj{m^{-1}}, (b+cj{m^{-1}})X, cX^2).
$$
By our assumption, the squared norm of $v$ must satisfy
$$
\|v\|^2 = (aN+bj{m^{-1}})^2 + (b+cj{m^{-1}})^2X^2 + c^2X^4 \leq \|v_0\|^2 = j^2 + m^2X^2.
$$
Consider the expression
\begin{align*}
|(am^2+bjtm-cj^2t)N|^2
&= \Big(m^2(aN+bj{m^{-1}}) - jm(b+cj{m^{-1}}) + cj^2\Big)^2\\
&\leq \Big((aN+bj{m^{-1}})^2+(b+cj{m^{-1}})^2X^2+c^2X^4\Big)\\
&\qquad\qquad\cdot\left(m^4+\frac{j^2m^2}{X^2}+\frac{j^4}{X^4}\right)
\qquad{\text{(by Cauchy--Schwarz)}}\\
&< (j^2+m^2X^2)\left(m^4+\frac{m^4}{X^2}+\frac{m^4}{X^4}\right)\\
&< 2N^{2\beta} \cdot 2m^4 < N^2,
\end{align*}
which implies that $am^2+bjtm-cj^2t=0$. Since $\gcd(j, m) = 1$ and $\gcd(t, m) = 1$, we must have $m \mid c$. However, we know that $|c| < 2N^{\beta}/X^2 < m$, which forces $c = 0$. 

Substituting this back into our equation yields $am+bjt=0$. Again, using the facts that $\gcd(j, m) = 1$ and $\gcd(t, m) = 1$, there must exist an integer $k$ such that $b = km$ and $a = -ktj$. This means that $v$ is an integer multiple of $v_0$, contradicting our assumption that $v$ is the shortest vector distinct from $v_0$.

We now prove that the second vector $v_2$ obtained from the LLL algorithm has a non-zero third coordinate. By Theorem \ref{thm:minkowski} and the fact that $\lambda_3 \geq \lambda_2 \geq \lambda_1 > N^{\beta}$, we have
\begin{equation}
\label{eq:lambda2}
\lambda_1\lambda_2\lambda_3 \leq (\sqrt{3})^3\det B \Rightarrow \lambda_2 \leq \left(\frac{3^{3/2}\det B}{\lambda_1}\right)^{1/2} < \frac{3^{3/4}N^{1/2+\beta}}{m^{3/2}}.
\end{equation}
From the properties of the LLL algorithm (Lemma \ref{lem:lll}), we know that
$$
\|v_2\| \leq 2\lambda_2 < \frac{2\cdot3^{3/4}N^{1/2+\beta}}{m^{3/2}}.
$$
If the third coordinate of $v_2$ were zero, then using the same analysis as above, we could write $v_2 = (aN + bj{m^{-1}}, bX, 0)$. This would give us
$$
(aN+bj{m^{-1}})^2+b^2X^2 = \|v_2\|^2 < \frac{4\cdot 3^{3/2}N^{1+2\beta}}{m^{3}}.
$$
Similarly,
\begin{align*}
|(am+bjt)N|^2 &= \Big(m(aN+bj{m^{-1}})-bj\Big)^2\\
&\leq \Big((aN+bj{m^{-1}})^2+b^2X^2\Big)\left(m^2+\frac{j^2}{X^2}\right)\\
&< \frac{4\cdot 3^{3/2}N^{1+2\beta}}{m^{3}} \cdot 2m^2\\
&= \frac{8\cdot 3^{3/2}N^{1+2\beta}}{m} < N^2,
\end{align*}
which implies $am+bjt=0$. Since $\gcd(j,m)=1$ and $\gcd(t,m)=1$, we must have $b=km$ and $a=-ktj$ for some integer $k$. But this would make $v_2$ an integer multiple of the shortest vector $v_0$, contradicting the linear independence of the basis vectors.
\end{proof}

Now we will prove the following proposition of our main search.

\begin{proposition}
\label{pro:main search}
Algorithm \ref{alg:search} is correct. It runs in time
\begin{equation}
\label{eq:search}
O\left(\varphi(m)\log^3 N + \frac{N^{1/2} \log^2 N}{m^{3/2}}\right).
\end{equation}
\end{proposition}

\begin{proof}
We first prove the correctness of the algorithm. {Write
\[
p = m\,p_{\mathrm{msb}} + p_{\mathrm{lsb}}, \qquad 0 \le p_{\mathrm{lsb}} < m,
\]
where $p_{\mathrm{msb}}$ and $p_{\mathrm{lsb}}$ denote, respectively, the most-significant and least-significant parts of $p$ in base $m$. Since $p$ is prime and $\gcd(m,N)=1$, we have $\gcd(m,p)=1$, hence $p_{\mathrm{lsb}}\in \mathbb{Z}_m^*$. In Step~\ref{step:jmgcd} we exhaust all $j\in\{1,\ldots,m\}$ with $\gcd(j,m)=1$, so the correct index $j=p_{\mathrm{lsb}}\equiv p \bmod m$ will be found.

Fix $j=p_{\mathrm{lsb}}$. We then use a Coppersmith-type construction combined with a BSGS search to recover $p_{\mathrm{msb}}$. Consider
\[
f(x)\coloneqq x+m^{-1}p_{\mathrm{lsb}}\bmod N,
\]
so that $x=p_{\mathrm{msb}}$ is a root of $f \bmod p$. The three-dimensional lattice we constructed in Step \ref{step:lattice-construction} is spanned by $N,f(x),f^2(x)$, thus $p_{\mathrm{msb}}$ is a root modulo $p$ of every polynomial encoded by a lattice vector, and in particular of $v_j=(c_j,\, b_jX,\, a_jX^2)$. Therefore
\[
p \mid c_j + b_j\,p_{\mathrm{msb}} + a_j\,p_{\mathrm{msb}}^{2}.
\]
}
Let $c_j + b_j\,p_{\mathrm{msb}} + a_j\,p_{\mathrm{msb}}^{2} = i p$. By the Cauchy–Schwarz inequality,
\begin{equation}\label{eq:cauchy}
i^2 p^2 \;=\; \big(c_j + b_j\,p_{\mathrm{msb}} + a_j\,p_{\mathrm{msb}}^{2}\big)^2
\;\le\; 3\big(c_j^2 + b_j^2\,p_{\mathrm{msb}}^2 + a_j^2\,p_{\mathrm{msb}}^{4}\big)
\;\le\; 3\|v_j\|^2.
\end{equation}
From \eqref{eq:lambda2} we have
\begin{equation}\label{eq:vj}
\|v_j\| \le 2\lambda_2 < \frac{2 \cdot 3^{3/4}\,N^{1/2+\beta}}{m^{3/2}}.
\end{equation}
Combining \eqref{eq:cauchy} and \eqref{eq:vj}, we obtain
\[
i \;<\; \frac{2 \cdot 3^{5/4}\,N^{1/2+\beta}}{m^{3/2}p}
\;<\; \frac{2 \cdot 3^{5/4}\,N^{1/2}}{c\,m^{3/2}}
\;\le\; k.
\]
Furthermore,
\[
m\,p_{\mathrm{msb}} = p - p_{\mathrm{lsb}} \equiv 1 - j \pmod{p-1}.
\]
Hence
\[
\alpha^{i m^2}
\equiv \alpha^{i p m^2}
\equiv \alpha^{m^2\big(c_j + b_j\,p_{\mathrm{msb}} + a_j\,p_{\mathrm{msb}}^{2}\big)}
\equiv \alpha^{c_j m^2 + b_j m(1-j) + a_j(1-j)^2}
= x_j \pmod{p}.
\]
Therefore there is a collision modulo $p$ between the babysteps computed in Step~\ref{step:babystep} and the giantsteps computed in Step~\ref{step:giantstep}.

If this collision is also a collision modulo $N$, then Step~\ref{step:match} identifies the pair $(i,j)$ and solves the corresponding quadratic. By Lemma~\ref{lem:second vector} we have $a_j\neq 0$, hence solving the quadratic returns $p$.

Otherwise, if the collision is only modulo $p$, then Step~\ref{step:collisions} recovers $p$.

We now analyze the running time. The number of babysteps is $k$, so the cost of Steps~\ref{step:babystep-loop}–\ref{step:babystep} is
\[
O\big(k \log N (\log \log N)^2\big).
\]
In Step~\ref{step:jmgcd} we compute $m$ gcds of integers bounded by $O(N)$, costing
\[
O\big(m \log N (\log \log N)^2\big).
\]
For the giantsteps, there are $\varphi(m)$ indices $j$. By Lemma~\ref{lem:lll}, LLL on our $3$-dimensional lattice costs $O(\log^{2+\varepsilon} N)$, so Steps~\ref{step:lattice-construction}–\ref{step:L1} cost
\[
O\big(\varphi(m)\log^{2+\varepsilon} N\big).
\]
Step~\ref{step:giantstep} can be carried out in time
\[
O\big(\varphi(m)\log N \log\log N\big).
\]
In Step~\ref{step:match}, we build a list of pairs $(\alpha^{m^2 i},i)$ of length $k$ and a list of tuples $(x_j, j)$ of length $O(\varphi(m))$. Each entry uses $O(\log N)$ bits. Sorting these lists by their first component via merge sort takes
\[
O\big((k + \varphi(m)) \log^2 N\big)
\]
bit operations. Since our input condition $\gcd(\alpha^{m^2 i}-1,N)=1$ for all $i\in[k]$ ensures that the babysteps are pairwise distinct, each $x_j$ matches with at most one babystep. Matching the two sorted lists hence takes
\[
O\big((k + \varphi(m)) \log N\big).
\]
There are at most $k$ matches, so solving the quadratic equations in Step~\ref{step:match} takes
\[
O\big(k \log N \log\log N\big).
\]
Finally, Step~\ref{step:collisions} invokes Algorithm~\ref{alg:collisions} with complexity
\[
O\big(\varphi(m)\log^3 N + k\log^2 N\big)
= O\!\left(\varphi(m)\log^3 N + \frac{N^{1/2}\log^2 N}{m^{3/2}}\right).
\]
Summing the costs of all steps, we obtain the overall bound
\[
O\!\left(\varphi(m)\log^3 N + \frac{N^{1/2}\log^2 N}{m^{3/2}}\right).
\]
\end{proof}

Then we present an algorithm to find the element $\alpha\in \mathbb{Z}_N^*$ for Algorithm \ref{alg:search}.

\begin{algorithm}[Finding $\alpha$]\ \\   % force line break
\label{alg:findalpha-bsgs}
\textit{Input:}
\begin{inputoutputlist}
\item A semiprime $N = pq$.
\item A positive integer $k$, and $m = O(N)$ with known factorization $m=\prod_{p_j\in S} p_j^{\alpha_j}$, where $S$ is a set of distinct primes.
\end{inputoutputlist}
\textit{Output:}
\begin{inputoutputlist}
\item Either an element $\alpha \in \mathbb{Z}_N^*$ such that for all $i \in \{1,\ldots,k\}$, $\gcd(\alpha^{m^2 i} - 1, N) = 1$; or the prime factors $p$ and $q$.
\end{inputoutputlist}
\begin{algorithmic}[1]
\State \label{step:hit2018-bsgs}
Apply Lemma \ref{lemma:improvehit2018} with $D \coloneqq \lceil N^{1/3}\rceil$.
If any factors of $N$ are found, return.
Otherwise, obtain $\alpha \in \mathbb{Z}_N^*$ with $\ord_N(\alpha) > D$.
\State \label{step:initialize} {Compute $\beta \gets \alpha^{m^2} \bmod N$, and $n \gets \lceil \sqrt{k} \rceil$.}
\State \label{step:comput} {Compute $\beta^{-1} \bmod N$ and the list of inverses $\beta^{-t} \bmod N$ for $t=1,2,\ldots,n$.}
\State \label{step:computpol} {Using Lemma \ref{lem:tree}, build the polynomial
$f(x) \gets \prod_{t=1}^{n}\bigl(x-\beta^{-t}\bigr) \in \mathbb{Z}_N[x]$.}
\State \label{step:evaluateall} {Using Lemma \ref{lem:bluestein}, evaluate
$y_i \gets f\!\left(\beta^{\,in}\right) \in \mathbb{Z}_N$ for $i=0,1,\ldots,n-1$, i.e., obtain $f(1), f(\beta^{n}), \ldots, f(\beta^{(n-1)n})$ modulo $N$.}
\For{$i=0,1,\ldots,n-1$} \label{step:scan-blocks}
    \State \label{step:block-gcd} {$g_i \gets \gcd(y_i, N)$.}
    \If{{$g_i \notin \{1, N\}$}} \label{step:block-factor}
        \State {Return $g_i$ and $N/g_i$.}
    \EndIf
    \If{{$g_i = N$}} \label{step:block-N} \Comment{{Localize within the block $B_i=\{in+1,\ldots,in+n\}$}}
        \For{{$j=1,2,\ldots,n$}} \label{step:scan-within-block}
            \State \label{step:delta-eval} {$\delta \gets \gcd\!\bigl(\beta^{\,in+j}-1, N\bigr)$.}
            \If{{$\delta \notin \{1, N\}$}} \label{step:delta-factor}
                \State {Return $\delta$ and $N/\delta$.}
            \EndIf
            \If{{$\delta = N$}} \label{step:delta-N}
                \State \label{step:ordreduce-init} $r \gets m^2\cdot (in+j)$, and factor $m^2 = \prod_{p_\ell \in S} p_\ell^{2\alpha_\ell}$.
                \For{each $p_\ell \in S$} \label{step:ordreduce-divloop}
                    \While{$p_\ell \mid r$ and $\gcd\!\bigl(\alpha^{\,r/p_\ell}-1, N\bigr) = N$} \label{step:ordreduce-divcond}
                        \State $r \gets r/p_\ell$ \label{step:ordreduce-div}
                    \EndWhile
                \EndFor
                \For{each $p_\ell \in S$ with $p_\ell \mid r$} \label{step:ordreduce-splitloop}
                    \State \label{step:ordreduce-split} $\delta' \gets \gcd\!\bigl(\alpha^{\,r/p_\ell}-1, N\bigr)$.
                    \If{$\delta' \notin \{1, N\}$} \label{step:ordreduce-splitfound}
                        \State Return $\delta'$ and $N/\delta'$.
                    \EndIf
                \EndFor
                \State \label{step:coro-factor} Using Corollary \ref{coro:pqmodm} with $p\equiv 1 \pmod r$ to factor $N$, and return $p$ and $q$.
            \EndIf
        \EndFor
    \EndIf
\EndFor
\State \label{step:return-alpha} Return $\alpha$.
\end{algorithmic}
\end{algorithm}

\begin{proposition}
\label{pro:findbeta}
{Algorithm \ref{alg:findalpha-bsgs}} is correct. It runs in time
$$
{O\left(\frac{N^{1/6}\log^2 N}{(\log\log N)^{1/2}}
+\sqrt{k}\,\log^3 N\right).}
$$
\end{proposition}

\begin{proof}
We first prove correctness. 
{Assume that the element $\alpha$ obtained in Step \ref{step:hit2018-bsgs} satisfies $\gcd(\alpha^{m^2i} - 1, N) = 1$ for all $i \in [n^2]$. Then the algorithm returns $\alpha$ in Step \ref{step:return-alpha}. Since $n^2>k$, this implies $\gcd(\alpha^{m^2i} - 1, N) = 1$ for all $i \in \{1,\ldots,k\}$.}

Otherwise, there exists some $i \in [n^2]$ such that $\gcd(\alpha^{m^2i} - 1, N) \neq 1$. Let $i_0$ be the smallest $i \in [n^2]$ such that $\gcd(\alpha^{m^2i} - 1, N) \neq 1$. If $\gcd(\alpha^{m^2i_0} - 1, N) \in \{p, q\}$, then the algorithm returns $p$ and $q$ in Step \ref{step:block-factor} (or Step \ref{step:delta-factor} if we are already inside the block-localization). The only remaining case is that $\gcd(\alpha^{m^2i_0} - 1, N) = N$.

Let $\ord_p(\alpha) = r_p$ and $\ord_q(\alpha) = r_q$. We first claim that $i_0\mid r_p$. We prove this by contradiction. Suppose $\gcd(i_0, r_p) < i_0$. Since
\[
r_p\mid i_0m^2 \Longrightarrow \frac{r_p}{\gcd(i_0, r_p)}\;\Big|\;\frac{i_0}{\gcd(i_0, r_p)}\,m^2,
\]
and $\gcd\!\bigl(r_p/\gcd(i_0, r_p),\, i_0/\gcd(i_0, r_p)\bigr) = 1$, we know that $(r_p/\gcd(i_0, r_p))\mid m^2$. Thus
\[
r_p = \gcd(i_0, r_p) \cdot \frac{r_p}{\gcd(i_0, r_p)}\;\Big|\;\gcd(i_0, r_p)\,m^2.
\]
This implies that $p\mid \gcd(\alpha^{\gcd(i_0, r_p)m^2} - 1, N)$, which contradicts the minimality of $i_0$. By similar reasoning, we can show that $i_0\mid r_q$.

At the beginning of the order-reduction stage (Step \ref{step:ordreduce-init}), we initialize $r = m^2 i_0$. The algorithm then performs a series of divisions by the prime factors of $m^2$. For each prime factor $p_j$, we divide $r$ by $p_j$ as long as $\gcd(\alpha^{r/p_j} - 1, N) = N$. This process continues until we find the smallest value $r$ such that $\alpha^r \equiv 1 \pmod{N}$.

We claim that after this process, $r$ equals $\ord_N(\alpha)$. This is because we start with a multiple of both $r_p$ and $r_q$ and gradually reduce it by removing unnecessary prime factors until we reach the minimum value that still satisfies $\alpha^r \equiv 1 \pmod{N}$.

Now, two cases are possible:

Case 1: $r_p \neq r_q$. Without loss of generality, there must exist a prime factor $p_j$ of $r_q$ such that $p_j \nmid r_p$ or $p_j$ appears with a higher exponent in $r_q$ than in $r_p$. When we compute $\gcd(\alpha^{r/p_j} - 1, N)$, we have $\alpha^{r/p_j} \equiv 1 \pmod{p}$ (since $r/p_j$ is still a multiple of $r_p$) and $\alpha^{r/p_j} \not\equiv 1 \pmod{q}$ (since $r/p_j$ is no longer a multiple of $r_q$). Therefore, $\gcd(\alpha^{r/p_j} - 1, N) = p$, and the algorithm returns $p$ and $q$ in Step \ref{step:ordreduce-splitfound}.

Case 2: $r_p = r_q = r$. In this case, we have $r = \ord_p(\alpha) = \ord_q(\alpha) = \ord_N(\alpha)$. Since $r > D > N^{1/4+o(1)}$ (from Step \ref{step:hit2018-bsgs}), we can apply Coppersmith's method as implemented in Corollary \ref{coro:pqmodm} to factor $N$ in polynomial time, and return $p$ and $q$ in Step \ref{step:coro-factor}.

This completes the proof of correctness.

We now analyze time complexity. The cost of Step \ref{step:hit2018-bsgs} from Lemma \ref{lemma:improvehit2018} is
\[
O\!\left(\frac{N^{1/6}\log^2 N}{(\log\log N)^{1/2}}\right).
\]

{To prepare the batched evaluation, we compute $\beta=\alpha^{m^2}\bmod N$ and $\beta^{-t}$ for $t=1,2,\dots,n$; this costs
$$
O\bigl((\sqrt{k}+\log m)\,\log N\,(\log \log N)^2\bigr).
$$
}
{We then form $f(x)=\prod_{t=1}^{n}(x-\beta^{-t})$ using Lemma \ref{lem:tree} and evaluate $y_i=f(\beta^{in})$ for $i=0,\dots,n-1$ by Lemma \ref{lem:bluestein} (Step \ref{step:evaluateall}), at total cost
$$
O\bigl(n\log^3 N\bigr)=O\bigl(\sqrt{k}\log^3 N\bigr).
$$
}
{The block GCDs (Step \ref{step:scan-blocks}) take $O(n)$ GCDs, i.e., $O\bigl(\sqrt{k}\,\log N\,(\log\log N)^2\bigr)$. If we enter the within-block localization (Steps \ref{step:block-N}–\ref{step:scan-within-block}), we perform at most $n=O(\sqrt{k})$ additional modular multiplications (advancing $\beta^{bn+j}$) and GCDs, for another $O\bigl(\sqrt{k}\,\log N\,(\log\log N)^2\bigr)$.}

If we go into the order-reduction stage (Steps \ref{step:ordreduce-init}–\ref{step:ordreduce-splitfound}), for each prime $p_j\in S$, we compute at most $2\alpha_j$ modular exponentiations and $2\alpha_j$ GCDs. Since $\sum \alpha_j \le \log m$, and the time cost of Coppersmith's method is polynomial in $\log N$, the total cost of Steps \ref{step:ordreduce-init}–\ref{step:coro-factor} is
\[
O\bigl(\log m \cdot \log N \, (\log\log N)^2\bigr) \;+\; \operatorname{poly}(\log N).
\]

{Summing up, the total running time is}
\[
{O\!\left(\frac{N^{1/6}\log^2 N}{(\log\log N)^{1/2}}
\;+\sqrt{k}\,\log^3 N\right),}
\]
which proves the stated bound.
\end{proof}

Finally we present the main factoring algorithm. In this algorithm, $N_0$ is a constant that is chosen large enough to ensure that the proof of correctness works for all $N \geq N_0$.

\begin{algorithm}[Factoring semiprimes]\ \\   % force line break
\label{alg:factor}
\textit{Input:} A semiprime $N =pq\ge N_0$ with $cN^{\beta}<p \le N^{\beta},1/3\le\beta\le 1/2$, where $c<1$ is a constant.
\noindent\textit{Output:}
$p$ and $q$.
\begin{algorithmic}[1]
\State\label{step:params}%
Compute 
$$
x\coloneqq\left\lfloor\frac{N^{1/5}(\log\log N)^{2/5}}{\log ^{2/5} N}\right\rfloor, 
$$
\State \label{step:m}
Apply Lemma \ref{lem:Prime Product Construction} with $x$ and get $m$ satisfying: 
$$
\frac{x}{2} < m < 2x, \quad \frac{\varphi(m)}{m} = \Theta\left( \frac{1}{\log\log N} \right).
$$
If $\gcd(N,m) \notin \{1, N\}$, recover $p$ and $q$ and return. Computing ${m^{-1}\pmod N}$ and 
$$
k \coloneqq \left\lceil \frac{2\cdot 3^{5/4}N^{1/2}}{cm^{3/2}}\right\rceil=\Theta\left(\frac{N^{1/5}\log^{3/5}N}{(\log\log N)^{3/5}}\right).
$$
\State\label{step:beta}%
Apply Algorithm \ref{alg:findalpha-bsgs} with $N,k,m$.
If any factors of $N$ are found, return. Otherwise, we obtain $\alpha \in \mathbb{Z}_N^*$ such that for all $i \in [k]$, $\gcd(\alpha^{m^2i} - 1, N) = 1$.
\State\label{step:search}%
Run Algorithm \ref{alg:search} (the main search)
with the given $N$, $m$, ${m^{-1}\pmod N}$ and $\alpha$.
Return $p$ and $q$.
\end{algorithmic}
\end{algorithm}

\begin{proposition}
Algorithm \ref{alg:factor} is correct (for suitable $N_0$),
and it runs in time
\[
O\left(\frac{N^{1/5} \log^{13/5} N}{(\log\log N)^{3/5}}\right).
\]
\end{proposition}

\begin{proof}
We first prove the correctness. Consider Step \ref{step:m}, we either get $p$ and $q$ or $m$ satisfying
$$
m=\Theta(x)=\Theta\left(\frac{N^{1/5}(\log\log N)^{2/5}}{\log ^{2/5} N}\right).
$$
In Step \ref{step:beta}, we either get $p$ and $q$ or obtain $\alpha \in \mathbb{Z}_N^*$ such that for all $i \in [k]$, $\gcd(\alpha^{m^2i} - 1, N) = 1$. Also, for large enough $N$, we have
$$
72<m<\frac{N^{1/4}}{2}\leq \frac{N^{(1-\beta)/2}}{2},
$$
which satisfies the requirement of Algorithm \ref{alg:search}.
In Step \ref{step:search}, we are guaranteed to get $p$ and $q$ by the Proposition \ref{pro:main search}.

Then we calculate the time complexity. Step \ref{step:params} and Step \ref{step:m} (from Lemma \ref{lem:Prime Product Construction}) are only polynomial time, which is negligible.

From Proposition \ref{pro:findbeta}, Step \ref{step:beta} costs
$$
O\left(\frac{N^{1/6}\log^2 N}{(\log\log N)^{1/2}} + {\sqrt{k}\,\log^3 N}\right),
$$
is also negligible.

From Proposition \ref{pro:main search},  Step \ref{step:search} costs
$$
O\left(\varphi(m)\log^3 N + \frac{N^{1/2} \log^2 N}{m^{3/2}}\right)=O\left(\frac{N^{1/5} \log^{13/5} N}{(\log\log N)^{3/5}}\right).
$$

\end{proof}

\begin{corollary}
Applying Algorithm \ref{alg:factor} with $\beta=1/2$, we prove Theorem \ref{thm:main}. 
\end{corollary}

\section{Better Factoring Sums and Differences of Powers}
\label{sec:anbn}
\subsection{Factoring Algorithm for Extra Modulo Information }

In the the previous section, we showed how the bit size of $p$ can help us to improve factoring. Now we focus on how to exploit extra modulo information.

We employ the same fundamental approach as in Algorithm \ref{alg:search}, with the primary distinction being the necessity to enumerate the possible bit sizes of $p$, resulting in a factor of $\log N$ increase in the number of giantsteps. Concurrently, the congruence constraint $p \equiv r \pmod n$ reduces the number of giantsteps by a factor of $n$.

\begin{algorithm}[The main search with modulo]\ \\   % force line break
\label{alg:searchwithmod}
\textit{Input:}
\begin{inputoutputlist}
\item A semiprime $N=pq,N^{1/3}<p<N^{1/2}$, $r,n$ such that $p\equiv r\pmod n, \gcd(n,N)=1$.

\item Positive integers $m\in \mathbb{Z}_N^*$, such that $(m,n)=1,72< mn < N^{1/4}/2, s=(mn)^{-1}\pmod{N},m_{inv}=m^{-1} \pmod n, n_{inv}=n^{-1}\pmod m$.

\item An element $\alpha\in\Z_N^*$ such that for all $i\in [k]$, $\gcd(\alpha^{(mn)^2i}-1,N)= 1$ where
$$
k \coloneqq \left\lceil \frac{4\cdot3^{5/4}N^{1/2}}{(mn)^{3/2}}\right\rceil.
$$
\end{inputoutputlist}
\textit{Output:}
$p$ and $q$.
\begin{algorithmic}[1]
\For{$i=0,\ldots,k-1$} \label{step:modbabystep-loop} \Comment{Computation of babysteps}
\State Compute $\alpha^{(mn)^2i} \pmod N$. \label{step:modbabystep}
\EndFor
\For{$i=1, \ldots, \left\lceil \frac{\log N}{6}\right\rceil$}\Comment{Computation of giantsteps}
    \State Compute $X_i=\left\lfloor \frac{2^i N^{0.3}}{mn}\right\rfloor$.
    \For{$j=1, \ldots, m$} 
        \If{$\gcd(j,m) = 1$} \label{step:modjmgcd}
            \State Compute $m_j\equiv r\,n\,n_{inv}+j\,m\,m_{inv}\pmod {mn}$.
            \State \label{step:modlattice-construction}
            Construct 3-rank lattice with basis
            \begin{equation}
            \label{eq:modlattice}
            B=\left(\begin{array}{ccc}
            N & 0 & 0\\
            m_js & X_i& 0\\
            m_j^2s^2 & 2m_jsX_i & X_i^2
            \end{array}\right).
            \end{equation}
            \State \label{step:swmL1}
            Apply Lemma \ref{lem:lll} (LLL Algorithm) and take the second vector
            \label{step:modsecond vector}
            $$
            v_{ij}=(c_{ij},b_{ij}X_i,a_{ij}X_i^2).
            $$
            \State \label{step:modgiantstep}
            Compute 
            $$
            x_{ij}=\alpha^{c_{ij}(mn)^2+b_{ij}mn(1-j)+a_{ij}(1-j)^2} \pmod N.
            $$
        \EndIf
    \EndFor
\EndFor
\State \label{step:modmatch}Applying a sort-and-match algorithm to the babysteps and giantsteps computed in Steps \ref{step:modbabystep} and \ref{step:modgiantstep}, find all pairs $(i, j,\sigma)$ such that
\begin{equation}
\label{eq:modmodn}
\alpha^{(mn)^2 \sigma} \equiv x_{ij} \pmod N.
\end{equation}
For each such match, solve the quadratic equation
$$
c_{ij}+b_{ij}\frac{(p-m_j)}{mn}+a_{ij}\frac{(p-m_j)^2}{(mn)^2}=\sigma p.
$$
If $p$ is found, return $p$ and $N/p$.

\State \label{step:modcollisions}%
Let $\{x_{ij}\}$ be the list of giantsteps computed in Step \ref{step:modgiantstep},
skipping those that were discovered in Step \ref{step:modmatch}
to be equal to one of the babysteps.
Apply Algorithm \ref{alg:collisions} (finding collisions)
with $N$, $\kappa \coloneqq k$,
${\gamma} \coloneqq \alpha^{(mn)^2} \pmod N$ and $\{x_{ij}\}$ as inputs.
\State \label{step:modreturn}%
Return $p$ and $q$.
\end{algorithmic}
\end{algorithm}

\begin{proposition}
\label{pro:modmain search}
Algorithm \ref{alg:searchwithmod} is correct. It runs in time
\begin{equation}
\label{eq:modsearch}
O\left(\varphi(m)\log^4 N + \frac{N^{1/2} \log^2 N}{(mn)^{3/2}}\right).
\end{equation}
\end{proposition}

\begin{proof}
The proof follows a similar structure to that of Proposition \ref{pro:main search}. We first establish the correctness of the algorithm. { Let
\[
p = mn\,p_{\mathrm{msb}} + p_{\mathrm{lsb}}, \qquad 0 \le p_{\mathrm{lsb}} < mn,
\]
where $p_{\mathrm{msb}}$ and $p_{\mathrm{lsb}}$ denote, respectively, the most-significant and least-significant parts of $p$ in base $mn$. Since $p$ is prime and $\gcd(mn,N)=1$, we have $\gcd(mn,p)=1$, hence $p_{\mathrm{lsb}}\in \mathbb{Z}_{mn}^*$. In Step~\ref{step:modjmgcd} we exhaust all $j\in\{1,\ldots,m\}$ with $\gcd(j,m)=1$, so the correct index $j_0=p_{\mathrm{lsb}}\equiv p \bmod m$ will be found.

For $j_0 = (p \bmod m)$, the Chinese remainder theorem implies that $p_{\mathrm{lsb}} = m_{j_0}$, which gives us $p | m_{j_0}s + p_{\mathrm{msb}}$. Consider $f(x)\coloneqq x+m_{j_0}s\bmod p$, and $x=p_{\mathrm{msb}}$ is a root of $f \bmod p$.

Setting $t = \left\lceil \log N/6\right\rceil$, we observe that $X_0 \leq |p_{\mathrm{msb}}| < \lfloor p/(mn) \rfloor \leq \lfloor N^{0.5}/(mn) \rfloor = X_t$. Therefore, there exists an index $i_0$ such that $X_{i_0-1} \leq |p_{\mathrm{lsb}}| \leq X_{i_0}$, which implies that $p \geq mnX_{i_0-1} = mnX_{i_0}/2$.

The three-dimensional lattice we constructed in Step \ref{step:modlattice-construction} is spanned by $N,f(x),f^2(x)$, thus $p_{\mathrm{msb}}$ is a root modulo $p$ of every polynomial encoded by a lattice vector, and in particular of $v_{i_0j_0} = (c_{i_0j_0}, b_{i_0j_0}X_{i_0}, a_{i_0j_0}X_{i_0}^2)$. Therefore
\[
p | c_{i_0j_0} + b_{i_0j_0}p_{\mathrm{msb}} + a_{i_0j_0}p^2_{\mathrm{msb}}.
\]
}
Let $c_{i_0j_0} + b_{i_0j_0}p_{\mathrm{msb}} + a_{i_0j_0}p^2_{\mathrm{msb}} = \sigma p$. By the Cauchy-Schwarz inequality, we have
\begin{equation}
\label{eq:modcauchy}
\sigma^2p^2 = (c_{i_0j_0} + b_{i_0j_0}p_{\mathrm{msb}} + a_{i_0j_0}p^2_{\mathrm{msb}})^2 \leq 3\|v_{i_0j_0}\|^2.
\end{equation}
Working with modulo $mn$, we note that for $i = 0, \ldots, t$, we have $\beta_i = \log_N(X_imn) \in [1/3, 1/2]$ and $mn < N^{1/4}/2 \leq N^{(1-\beta_i)/2}/2$, which satisfies the conditions of Lemma \ref{lem:second vector}.

From equation \eqref{eq:lambda2}, we derive 
\begin{equation}
\label{eq:modvj}
\|v_{i_0j_0}\| \leq 2\lambda_2 < \frac{2 \cdot 3^{3/4}N^{1/2+\beta_{i_0}}}{(mn)^{3/2}}.
\end{equation}
Combining inequalities \eqref{eq:modcauchy} and \eqref{eq:modvj}, we obtain
$$
\sigma < \frac{2 \cdot 3^{5/4}N^{1/2+\beta_{i_0}}}{(mn)^{3/2}p} \leq \frac{4 \cdot 3^{5/4}N^{1/2}}{(mn)^{3/2}} \leq k.
$$
Following the same reasoning as in Proposition \ref{pro:main search}, we have $\alpha^{\sigma (mn)^2} \equiv x_{i_0j_0} \pmod{p}$. So there must exist a collision modulo $p$ between the babysteps and the giantsteps.

If this collision also manifests as a collision modulo $N$, then Step \ref{step:modmatch} will identify the triplet $(i_0, j_0, \sigma)$ and solve the corresponding quadratic equation. By Lemma \ref{lem:second vector}, we know that $a_{i_0j_0} \neq 0$, ensuring that we can successfully recover $p$ by solving the quadratic equation.

Alternatively, if the collision only occurs modulo $p$ but not modulo $N$, the algorithm will still determine $p$ in Step \ref{step:modcollisions}.

We now analyze the time complexity of the algorithm. The number of babysteps is $k$, and the number of giantsteps is $t\varphi(m)$. Therefore, the computational cost of finding collisions between babysteps and giantsteps is
$$
O(k \log^2 N + t\varphi(m)\log^3 N) = O\left(\varphi(m)\log^4 N + \frac{N^{1/2} \log^2 N}{(mn)^{3/2}}\right).
$$
The time complexity of the remaining steps is negligible compared to this dominant term, analogous to the analysis in Proposition \ref{pro:main search}, which completes our proof.
\end{proof}

Now we revisit Theorem 2.8 in \cite{hittmeir2017deterministic} with our three dimensional-lattice technique. 
\begin{theorem}
\label{thm:N1/5with-n}
{
Given an integer $N \ge 2$ and integers $n \ge 1$ and $r$ such that $p \equiv r \pmod n$ for every prime divisor $p$ of $N$, one can compute the prime factorization of $N$ in
\[
O\!\left(\frac{N^{1/5}\log^{16/5}N}{n^{3/5}}+N^{1/6}\log^4 N\right).
\]
}
\end{theorem}

\begin{proof}
We first note that if $N$ is prime, we can identify this using the AKS algorithm in polynomial time.

If $N$ has three or more prime factors (counting repetitions), then at least one factor is bounded above by $N^{1/3}$, and such a factor may be found in time $O(N^{1/6}\log^3N)$ (see \cite[Prop.~2.5]{harvey2021exponent}). By repeating this process at most $\log N$ times, we reduce to the case where $N$ is a semiprime with $p,q>N^{1/3}$. The cost of this reduction step is at most 
$$
O\left(N^{1/6}\log^4 N\right).
$$

{
If $\gcd(n,N)>1$, let $p_0 \mid \gcd(n,N)$ be a prime. By the hypothesis that $p \equiv r \pmod n$ for every prime $p \mid N$, taking $p=p_0$ gives $p_0 \equiv r \pmod n$. Since $p_0 \mid n$, reducing modulo $p_0$ yields $r \equiv 0 \pmod{p_0}$. Hence for any prime $p \mid N$ we have $p \equiv r \equiv 0 \pmod{p_0}$, so $p=p_0$ and therefore $N$ is a power of $p_0$. In this case we can factor $N$ in polynomial time. Thus we may assume $\gcd(n,N)=1$ in what follows.
}
Then we assume $n<N^{1/4}$; otherwise, we could apply Corollary \ref{coro:pqmodm} to factor $N$ in polynomial time.

Next, we determine the parameter $m$. We begin by computing $m=\left\lceil \frac{N^{1/5}}{\log^{4/5}N\cdot n^{3/5}}\right\rceil$. If $\gcd(m,n)>1$, we set $m=m+1$ and iterate until we find a value such that $\gcd(m,n)=1$. We claim that this process examines at most $O(\log^2n)$ values of $m$. This follows from a result by Iwaniec \cite{iwaniec1978problem}, which states that $j(n)=O(\log^2n)$, where $j(n)$ is defined as the smallest positive integer $y$ such that every sequence of $y$ consecutive integers contains an integer coprime to $n$. Since $\log^2n<\log^2N\ll \frac{N^{1/5}}{\log^{4/5}N\cdot n^{3/5}}$, the value of $m$ we obtain satisfies 
$$
m=\Theta\left(\frac{N^{1/5}}{\log^{4/5}N\cdot n^{3/5}}\right), \quad \gcd(m,n)=1,
$$
and the computational complexity of this step is only polynomial in $\log N$, which is negligible in the overall complexity analysis.

We proceed to compute $s=(mn)^{-1}\pmod{N}$, $m_{inv}=m^{-1} \pmod n$, $n_{inv}=n^{-1}\pmod m$, and the factorization of $mn$. These computations can be performed in polynomial time, and the factorization costs at most $(mn)^{1/4}=O((N^{3/10})^{1/4})=O(N^{3/40})$, which is also negligible.

Let
$$
k \coloneqq \left\lceil \frac{4\cdot3^{5/4}N^{1/2}}{(mn)^{3/2}}\right\rceil=\Theta\left(\frac{N^{1/5}\log^{6/5}N}{n^{3/5}}\right).
$$

We apply Algorithm \ref{alg:findalpha-bsgs} with parameters $N$, $k$, and $mn$, which either yields $p$ and $q$ directly or provides an element $\alpha \in \mathbb{Z}_N^*$ such that for all $i \in [k]$, $\gcd(\alpha^{(mn)^2i} - 1, N) = 1$. This step incurs a cost of
$$
O\left(\frac{N^{1/6}\log^2 N}{(\log\log N)^{1/2}} + {\sqrt{k}\,\log^3 N}\right),
$$
which is negligible in the overall complexity.

Finally, we employ Algorithm \ref{alg:searchwithmod} with the aforementioned parameters to factor the semiprime $N$. We can verify that all conditions required by Algorithm \ref{alg:searchwithmod} are satisfied. The complexity of this algorithm is
$$
O\left(\varphi(m)\log^4 N + \frac{N^{1/2} \log^2 N}{(mn)^{3/2}}\right)=O\left(\frac{N^{1/5}\log^{16/5}N}{ n^{3/5}}\right).
$$

Combining all the above analyses, the total complexity of the algorithm is 
$$
O\left(\frac{N^{1/5}\log^{16/5}N}{n^{3/5}}+N^{1/6}\log^4 N\right).
$$
\end{proof}

\begin{remark}
When $n$ is not a product of small primes, by selecting an appropriate $m$ as a product of small primes, we can obtain the same loglog speedup as Algorithm \ref{alg:factor}, yielding an improved complexity of $O\left(\frac{N^{1/5}\log^{16/5}N}{(n\log\log N)^{3/5}}\right)$ in the first term of Theorem \ref{thm:N1/5with-n}.
\end{remark}

\begin{remark}
{
Write $m = N^a$. Ignoring $N^{o(1)}$ factors, Theorem \ref{thm:prqfactorsmall} runs in time $N^{1/4 - a}$, while Theorem \ref{thm:N1/5with-n} runs in $\max\{N^{1/6},\, N^{1/5 - 3a/5}\}$.  Under this bound, the crossover with $N^{1/4 - a}$ occurs at $a = 1/12$: Theorem \ref{thm:N1/5with-n} is superior for $a < 1/12$ (equal at $a = 1/12$), and Theorem \ref{thm:prqfactorsmall} dominates for $a > 1/12$. Note also the modeling difference: Theorem \ref{thm:prqfactorsmall} allows isolating only those prime factors $p \equiv r \pmod m$ and assumes $\gcd(m, N) = 1$, whereas Theorem \ref{thm:N1/5with-n} assumes that every prime factor $p$ of $N$ satisfies $p\equiv r \pmod m$; see Figure~\ref{fig:comparison}.
}
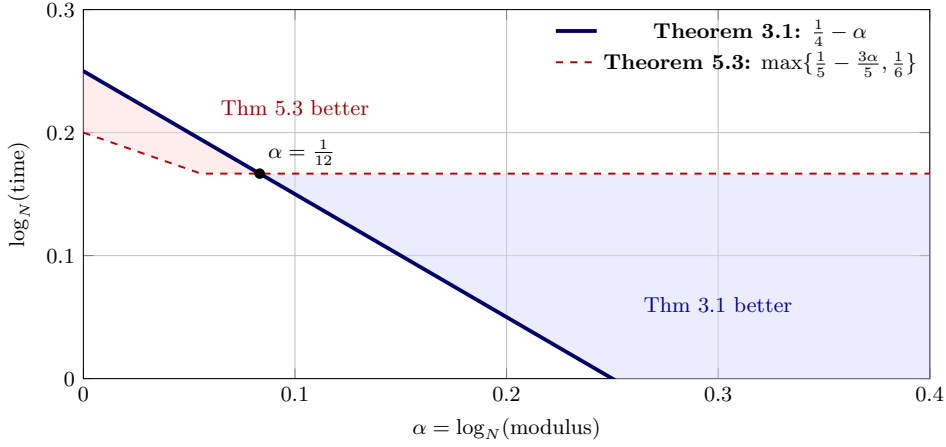
\begin{figure}[htbp]
  \centering
  \begin{tikzpicture}
  \begin{axis}[
    xlabel={$\alpha = \log_N(\mathrm{modulus})$},
    ylabel={$\log_N(\mathrm{time})$},
    xmin=0, xmax=0.4,
    ymin=0, ymax=0.3,
    xtick={0,0.1,0.2,0.3,0.4},
    ytick={0,0.1,0.2,0.3},
    grid=both,
    width=14cm,
    height=7cm,
    legend style={
      at={(1,1)}, anchor=north east,
      draw=none, fill=none, font=\small
    },
    label style={font=\small},
    tick label style={font=\small},
    domain=0:0.4,
    samples=200,
    thick
  ]

  % --- Curves ---
  \addplot[name path=thm31, blue!40!black, ultra thick] {1/4 - x};
  \addlegendentry{\textbf{Theorem 3.1:} $\tfrac{1}{4} - \alpha$}

  % Theorem 5.3: modified to max{1/5 - 3*x/5, 1/6}
  \addplot[name path=thm53, red!70!black, dashed, thick, domain=0:0.4]
    ({x},{max(1/5 - 3*x/5, 1/6)});
  \addlegendentry{\textbf{Theorem 5.3:} $\max\{\tfrac{1}{5} - \tfrac{3\alpha}{5}, \tfrac{1}{6}\}$}

  % --- Shaded regions ---
  \addplot[fill=red!15, opacity=0.5] fill between[
      of=thm31 and thm53,
      soft clip={domain=0:1/12}
  ];
  \addplot[fill=blue!15, opacity=0.5] fill between[
      of=thm31 and thm53,
      soft clip={domain=1/12:0.4}
  ];

  % --- Intersection point ---
  \addplot[mark=*, mark size=1.8pt, black] coordinates {(1/12,1/6)};
  \node[above right,font=\small] at (axis cs:1/12,1/6)
    {$\alpha=\tfrac{1}{12}$};

  % --- Text annotations ---
  \node[red!60!black,font=\small] at (axis cs:0.10,0.22)
    {Thm 5.3 better};
  \node[blue!60!black,font=\small] at (axis cs:0.3,0.06)
    {Thm 3.1 better};

  \end{axis}
\end{tikzpicture}
  \caption{Comparison of Theorem 3.1 and Theorem 5.3}
  \label{fig:comparison}
\end{figure}

\end{remark}

\subsection{Factoring Sums and Differences of Powers}
While maintaining the original algorithmic framework proposed by Hittmeir (Algorithm 3.1 in \cite{hittmeir2017deterministic}), we present an improved version of his Theorem 2.8 as Theorem \ref{thm:N1/5with-n} for $n=\Theta(\log N)$, which yields better bounds in the final results.

\begin{algorithm}[Factoring sums and differences of powers]\ \\
\label{alg:factoranbn}
\textit{Input:}
\begin{inputoutputlist}
\item Coprime integers $a, b \in \mathbb{N}$ with $a > b$.
\item A number $N \in P_{a,b}$ where either $N = a^n + b^n \in P_{a,b}^+$ or $N = a^n - b^n \in P_{a,b}^-$ for some $n \in \mathbb{N}$.
\end{inputoutputlist}
\textit{Output:}
\begin{inputoutputlist}
\item The prime factorization of $N$.
\end{inputoutputlist}
\begin{algorithmic}[1]
\State Set $N_1 := N, v := 1$.
\State Apply trial division to compute all divisors of $n$.
\If{$N \in P_{a,b}^+$}
    \State Define $\mathcal{D} := \{2d: d \mid n\}$.
\Else
    \State Define $\mathcal{D} := \{d: d \mid n\}$.
\EndIf
\State Let $d_1 < d_2 < \cdots < d_l$ be the ordered list of all elements in $\mathcal{D}$ where $l \geq 2$.
\While{$N_j \neq 1$}
    \State Set $j = v$.
    \State Compute $G_j = \gcd((ab^{-1})^{d_j} - 1 \bmod N_j, N_j)$.
    \If{$G_j = 1$}
        \State Set $N_{j+1} = N_j$.
    \ElsIf{$1 < G_j \leq N$}
        \State Apply Theorem \ref{thm:N1/5with-n} with $n = d_j$ and $r = 1$ to compute prime factorization of $G_j$.
        \State Remove all prime factors dividing $G_j$ from $N_j$ to obtain $N_{j+1}$.
    \EndIf
    \State Set $v = j+1$.
\EndWhile
\State Return the prime factorization of $N$.
\end{algorithmic}
\end{algorithm}

\begin{proposition}
Algorithm \ref{alg:factoranbn} is correct and it runs in time
\[
O\left(N^{1/5} \log^{13/5} N\right).
\]    
\end{proposition}
\begin{proof}
The correctness of Algorithm \ref{alg:factoranbn} has already been proved in proof of Theorem 1.1 \cite{hittmeir2017deterministic}. We now focus on the time complexity.

Following the analysis in \cite{hittmeir2017deterministic}, we have $n = O(\log N)$, and for $j < l$, we obtain $G_j \leq N^{1/2}$ when $N \in P_{a,b}^-$ and $G_j \leq N^{2/3}$ when $N \in P_{a,b}^+$. For $j = l$, note that $n = \Theta(\log N)$. In the worst case where $G_l = N$, applying Theorem \ref{thm:N1/5with-n} yields a runtime of 
$$
O\left(\frac{N^{1/5}\log^{16/5}N}{n^{3/5}}+N^{1/6}\log^4 N\right)=O(N^{1/5}\log^{13/5}N),
$$
which dominates all other operations and gives the claimed complexity bound. This also proves Theorem \ref{thm:anbn}.
\end{proof}
\begin{remark}
For the case where $j = l$, we have $n = \Theta(\log N)$. In this scenario, one can achieve a logarithmic speedup by carefully selecting {small primes} that are coprime to $n$, thus obtaining a {modulus} $m$ that satisfies both $\gcd(m,n)=1$ and $\varphi(m)/m=\Theta(1/\log\log N)$. This leads to the complexity bound $O\left(\frac{N^{1/5}\log^{13/5}N}{(\log\log N)^{3/5}}\right)$. 

Notably, this matches the complexity in Theorem \ref{thm:main} for balanced semiprimes, where the bit-length information of $p$ saves a factor of $\log N$ in the exhaustive search, while here the additional information from $n = \Theta(\log N)$ combined with the $\log\log$ speedup achieves the same effect.
\end{remark}

\section{Speedup for $p^rq$}
\label{sec:prq}
Finally, we generalize our rank-3 lattice construction to solve the $r$-power divisor problem. This problem—finding all integers $p$ such that $p^r \mid N$—has recently seen notable progress in the development of provably deterministic algorithms. Hales and Hiary~\cite{hales2024generalization} extended Lehman's method~\cite{lehman1974factoring} and obtained two algorithms with complexities 
\[
O\left(N^{1/(r+2)} (\log N)^2 \log \log N\right) \quad \text{and} \quad O\left(N^{1/(3+2r)} (\log N)^{16/5}\right).
\]

Around the same time, Harvey and Hittmeir~\cite{harvey2022deterministic} (Proceedings of ANTS XV, \emph{Res. Number Theory} 8 (2022), no.~4, Paper No.~94) applied Coppersmith's method directly and achieved a complexity of 
\[
O\left( \frac{N^{1/4r} \log^{10+\epsilon} N}{r^3} \right).
\]
By incorporating faster LLL-type lattice reduction algorithms and sieving on small primes, we improve this to 
\[
O\left( \frac{N^{1/4r} \log^{7+3\epsilon} N}{(\log\log N-\log 4r)r^{2+\epsilon}} \right).
\]

According to Remark~3.5 in~\cite{harvey2022deterministic}, the worst-case running time of their algorithm occurs when $p=\Theta(N^{1/2r})$ and $q =\Theta(N^{1/2})$, that is, when $N = p^r q$. By focusing on this case and employing our rank-3 lattice construction, we further reduce the complexity to
\[
O\left({r^{1/4}} N^{1/4r} \log^{5/2} N \right).
\]

\subsection{More Refined Complexity Analysis of~\cite{harvey2022deterministic}}

Set $m=1$ in Theorem \ref{thm:prqfactorsmall}, we first demonstrate how the complexity bound in~\cite{harvey2022deterministic} can be improved from 
\[
O\left( \frac{N^{1/4r} \log^{10+\epsilon} N}{r^3} \right)
\quad \text{to} \quad
O\left( \frac{N^{1/4r} \log^{7+3\epsilon} N}{r^{2+\epsilon}} \right).
\]
\begin{theorem}
\label{thm:prq7+epsilon}
There exists an explicit deterministic algorithm with the following properties. Given as input an integer $N > 2$ and a positive integer $r$, the algorithm outputs a list of all positive integers $p$ such that $p^r \mid N$. Its running time is
\[
O\left( \frac{N^{1/4r} \log^{7+3\epsilon} N}{r^{2+\epsilon}} \right).
\]
\end{theorem}
\begin{proof}
Set $s=0$ and $m=1$ in Theorem~\ref{thm:prqfactorsmall}.
\end{proof}
\begin{remark}
We emphasize that this improvement does not arise from a novel algorithmic idea, but rather from a more refined complexity analysis combined with the use of alternative lattice basis reduction algorithms~\cite{NS2016fasterlll}.
\end{remark}

We could also get a faster algorithm using the same idea of sieving on the small primes.
\begin{theorem}
\label{thm:prq7+epsilonwithspeedup}
Let $N$ be a natural number, one can compute all primes $p$ such that $p^r | N$ in time
$$
F(N)=O\left(\frac{ N^{1/4r}} {\log\log N-\log 4r} \frac{\log^{7+3\epsilon} N}{r^{2+\epsilon}} \right).
$$

\end{theorem}

\begin{proof}
Set $x=\lceil N^{1/4r}\rceil$ and apply Lemma \ref{lem:Prime Product Construction} to obtain $m$ satisfying
$$
m = \prod_{p \in S} p,\quad \frac{x}{2} < m < 2x, \quad \frac{\varphi(m)}{m} = \Theta\left( \frac{1}{\log\log x} \right)
$$
This step requires only polynomial time in $\log N$.

We first compute $G=\gcd(m,N)$, then identify all primes $p \in S$ such that $p | G$, and verify for each whether $p^r | N$. We remove all prime factors $p$ of $G$ from $N$ to obtain $N_1$.

Since $|S| = O(\log N)$, this step also requires only polynomial time in $\log N$.

Next, we iterate through all elements $i \in \mathbb{Z}_m^*$. For each $i$, we apply Theorem \ref{thm:prqfactorsmall} with parameters $N_1, r, s=i$, and $m$ to find all primes $p$ satisfying $p \equiv i \pmod{m}$ and $p^r | N_1$. We claim this will identify all primes $p$ such that $p^r | N_1$. This is because if a prime $p$ divides $N_1$, then $\gcd(p,m) = 1$, which implies $(p \bmod m) \in \mathbb{Z}_m^*$. Therefore, when $i = (p \bmod m)$, this prime $p$ will be output by Theorem \ref{thm:prqfactorsmall}. Combined with our earlier discussion, we prove that we output all primes $p$ such that $p^r | N$.

The time complexity for each $i$ is
$$
O\left(\left\lceil\frac{N_1^{1/4r}}{m}\right\rceil \frac{\log^{7+3\epsilon} N}{r^{2+\epsilon}} \right) = O\left(\frac{\log^{7+3\epsilon} N}{r^{2+\epsilon}} \right)
$$
Since there are $\varphi(m) = O(m/\log\log x) = O(N^{1/4r}/(\log\log N - \log 4r))$ values of $i$, the total time complexity is
$$
O\left(\frac{N^{1/4r}}{\log\log N - \log 4r} \cdot \frac{\log^{7+3\epsilon} N}{r^{2+\epsilon}} \right)
$$
\end{proof}
\begin{remark}
Compared to Theorem~\ref{thm:prq7+epsilon}, Theorem~\ref{thm:prq7+epsilonwithspeedup} does not always yield a genuine $\log \log N$-speedup. When $r = O(1)$, the improvement $\log \log N - \log 4r$ can indeed be interpreted as a $\log \log N$-speedup. However, in the case where $r = \Theta(\log N / \log \log N)$, the difference becomes $\Theta(\log \log \log N)$, which is asymptotically weaker. A similar issue arises in Proposition~4.4 of~\cite{hales2024generalization}, where the authors claim that “a more sophisticated choice of $m$ can give a $\log \log N$ speedup.” This assertion is not entirely accurate in general.

\end{remark}

\subsection{Factor $N = p^r q$ with $q=\Theta(N^{1/2})$}

Next, we focus on this specific scenario and demonstrate how to use our rank-3 lattice to solve the $r$-power divisor problem. For $r>\frac{\log N}{32\log\log N}$, then $p<\log^{1/64}N$, which means we can find $p$ in $\text{poly}(\log N)$ time by enumeration. Then we consider $r<\frac{\log N}{32\log\log N}$.

We now present the main search algorithm:

\begin{algorithm}[The main search for $N=p^r q$]\ \\   % force line break
\label{alg:searchprq}
\textit{Input:}
\begin{inputoutputlist}
\item An integer $N=p^{r}q$ with {constants $0<c_{1}<c_{2}$ such that $c_{1}N^{1/2}<q<c_{2}N^{1/2}$}.
\item Positive integer {$m<p/r$ and set $X\coloneqq m$}.
\item An element $\alpha\in\Z_N^*$ whose order larger than {$k \coloneqq \left\lceil 2em\sqrt{r}\right\rceil.$} 
\end{inputoutputlist}
\textit{Output:}
$p$ and $q$.
\begin{algorithmic}[1]
\For{$i=0,\dots,k-1$}  \Comment{Computation of babysteps}\label{step:babystepprq}
\State Compute $\alpha^{i} \pmod N$ and $\gcd(\alpha^{i}-1,N)$.
\If{$\gcd(\alpha^{i}-1,N)>1$}
\State Compute $p$ and $q$ and return $p$ and $q$.
\EndIf
\EndFor

\For{$j=0, \ldots, \left\lfloor\frac{N^{1/2r}}{X\cdot (c_1)^{1/r}} \right\rfloor$} \Comment{Computation of giantsteps}\label{step:giantstepexhaustprq}
    \State $M_j \coloneqq m\,j$, and {define $f_j(x)\coloneqq x+M_j$}. \label{step:fj-def}
    \State \label{step:lattice-constructionprq}
    {
    Construct the $3$-rank lattice with basis $B_j\in\mathbb{Z}^{3\times (r+2)}$, where the $i$-th column is the coefficient of $x^{i}$ in the polynomials $N,f^r_j(xX),f^{r+1}_j(xX)$ :
    \begin{equation*}
    \label{eq:latticeprq}
    B_j=\begin{pmatrix}
        N & 0 & 0 & \cdots & 0 & 0\\
        \binom{r}{0}M_j^{r} & \binom{r}{1}M_j^{r-1}X & \binom{r}{2}M_j^{r-2}X^2& \cdots  & \binom{r}{r}X^{r} & 0\\
        \binom{r+1}{0}M_j^{r+1} & \binom{r+1}{1}M_j^{r}X & \binom{r+1}{2}M_j^{r-1}X^2 & \cdots & \binom{r+1}{r}M_j X^{r} & \binom{r+1}{r+1}X^{r+1}
        \end{pmatrix}.
    \end{equation*}
    \State \label{step:vectorprq}
    Define the difference vector:
    \begin{align*}
    v_j &=(v_{j}^{0},\,v_{j}^{1}X,\,\dots,\,v_{j}^{r+1}X^{r+1})\\
    &:= \big(0,\ \tbinom{r}{0}M_j^{r}X,\ \tbinom{r}{1}M_j^{r-1}X^2,\ \dots,\ \tbinom{r}{r-1}M_j X^{r},\ \tbinom{r}{r}X^{r+1}\big).
    \end{align*}}
    \State 
    Let the associated polynomial be
    \[
      g_{j}(x)=v_{j}^{0}+v_{j}^{1}x+\cdots+v_{j}^{r+1}x^{r+1}\in\mathbb{Z}[x].
    \]
    \State \label{step:giantstepprq}
    Compute the giantstep
    \[
      x_{j}\coloneqq \alpha^{\,g_{j}(1-M_j)} \pmod N,
    \]
    and store the pair $(x_j,\,j)$.
\EndFor

\State \label{step:matchprq}%
Applying a sort-and-match algorithm to the babysteps and giantsteps
computed in Steps \ref{step:babystepprq} and \ref{step:giantstepprq},
find all pairs $(i, j)$ such that
\begin{equation*}
\label{eq:modn-msb}
\alpha^{\,i} \equiv x_{j} \pmod N.
\end{equation*}
For each such match, solve the {degree-$(r+1)$} equation in $p$:
\[
\sum_{t=0}^{r+1} v_{j}^{t}\left(p-M_j\right)^{t} \;=\; i\,p^{\,r}.
\]
If an integer solution $p$ is found, return $p$ and $q$.

\State \label{step:collisionsprq}%
Let $x_{1},\ldots, x_{n}$ be the list of giantsteps computed in Step \ref{step:giantstepprq},
skipping those that were discovered in Step \ref{step:matchprq}
to be equal to one of the babysteps.
Apply Algorithm \ref{alg:collisions} (finding collisions)
with $N$, $\kappa \coloneqq k$, \(\gamma \coloneqq \alpha \pmod N\) and $x_{1},\ldots, x_{n}$ as inputs.

\State \label{step:mainreturnprq}%
Return $p$ and $q$.
\end{algorithmic}
\end{algorithm}

Now we will prove the following proposition of our main search.

\begin{proposition}
\label{pro:prqmain search}
Algorithm \ref{alg:searchprq} is correct. It runs in time
\begin{equation*}
\label{eq:prqsearch}
O\left({\sqrt{r}\,m\log^2N+\frac{N^{1/2r}}{m}\log^3N}\right).
\end{equation*}
\end{proposition}
\begin{proof}
The correctness and most parts of the complexity analysis follow similarly to the proof of Proposition~\ref{pro:main search}. Write
\[
p = m\,p_{\mathrm{msb}} + p_{\mathrm{lsb}}, \qquad 0 \le p_{\mathrm{lsb}} < m=X,
\]
where $p_{\mathrm{msb}}$ and $p_{\mathrm{lsb}}$ denote the most-significant and least-significant parts of $p$ in base $m$. In Step~\ref{step:giantstepexhaustprq} we exhaust all $j\le p/X$, so the correct index $j=p_{\mathrm{msb}}$ will be found.

Fix $j=p_{\mathrm{msb}}$, we then define $f_j(x)\coloneqq x+M_j$, thus $x=p_{\mathrm{lsb}}$ is a small root of $f_j(x) \bmod p$ and $|x|<m=X$ holds.

We then construct the three-dimensional lattice $B_j$ spanned by $N,f^r_j(x),f^{r+1}_j(x)$, thus $x_0=p_{\mathrm{lsb}}$ is a root modulo $p^r$ of every polynomial encoded by a lattice vector, and in particular of $v_{j}=(v_{j}^{0},\,v_{j}^{1}X,\,\dots,\,v_{j}^{r+1}X^{r+1})$. Therefore $p^r|g_j(x_0)$.

Let $i p^r=g_j(x_0)$, we could bound the norm of $v_j$:
\begin{align*}
\|v_j\|
&= \sqrt{\sum_{i=0}^r \left(\binom{r}{i} M_j^i X^{r+1-i} \right)^2} \\
&\leq \sqrt{\left(\sum_{i=0}^r \binom{r}{i} M_j^i X^{r+1-i} \right)^2} \\
&= X(X + M_j)^r \\
&= X p^r \left(1 + \frac{X-p_{\mathrm{lsb}}}{p} \right)^r \\
&<X p^r \left(1 + \frac{1}{r} \right)^r\\
&<e X p^r .
\end{align*}

By the Cauchy–Schwarz inequality,
\begin{equation*}
ip^r=g_j(x_0)\le \sqrt{(r+2)\|v_j\|^2}<2\sqrt{r}eXp^r
\end{equation*}
Then we obtain
\[
i < 2em\sqrt{r}\le k.
\]
Furthermore,
\[
x_0=p_{\mathrm{lsb}}\equiv 1-mp_{\mathrm{msb}}=1-M_j\bmod (p-1).
\]
Hence
\[
\alpha^{i }
\equiv\alpha^{i p^r}
\equiv \alpha^{g_j(x_0)}
\equiv \alpha^{g_j(1-M_j)}
= x_j \pmod{p}.
\]
Therefore there is a collision modulo $p$ between the babysteps computed in Step~\ref{step:babystepprq} and the giantsteps computed in Step~\ref{step:giantstepprq}.

If this collision is also a collision modulo $N$, then Step~\ref{step:matchprq} identifies the pair $(i,j)$ and solves the corresponding equation. We have $v^{r+1}_j=1\neq 0$, hence solving the equation returns $p$.

Otherwise, if the collision is only modulo $p$, then Step~\ref{step:collisionsprq} also recovers $p$. 

{

We now analyze the running time. From Theorem~1.2 in \cite{harvey2022deterministic}, the cost of solving a polynomial $f \in \mathbb{Z}[x]$ of degree $r+1$ is $O(r^{2+\epsilon}\log^{1+\epsilon}(\|f\|_{\infty}))=O(r^{2+\epsilon}\log^{1+\epsilon}(\|v_j\|))=O(r^{2+\epsilon}\log^{1+\epsilon}N)$. Thus, solving the degree-$(r+1)$ equations in Step~\ref{step:matchprq} takes $O\big(k r^{2+\epsilon}\log^{1+\epsilon}N\big)$, which is negligible. As in the proof of Proposition~\ref{pro:main search}, in other steps, the dominant cost comes from applying Algorithm \ref{alg:collisions} (finding collisions) in step \ref{step:collisionsprq}. There are $k$ baby steps and
\[
\left\lfloor\frac{N^{1/(2r)}}{X\,(c_1)^{1/r}}\right\rfloor
= \Theta\!\left(\frac{N^{1/(2r)}}{m}\right)
\]
giant steps. By Proposition~\ref{prop:findCollisions}, the total running time is
\[
O\!\left(k\,\log^2\!N+\frac{N^{1/(2r)}}{m}\,\log^3\!N\right)
=O\!\left(\sqrt{r}\,m\,\log^2\!N+\frac{N^{1/(2r)}}{m}\,\log^3\!N\right).
\]
}
\end{proof}

{

\begin{remark}
In Step~\ref{step:vectorprq}, LLL reduction is not needed. The $3$-dimensional lattice
$L_j=\langle N,\ f_j(xX)^r,\ f_j(xX)^{r+1}\rangle\subset\mathbb{Z}^{r+2}$ with
$f_j(x)=x+M_j$ already contains two explicit short vectors that we can write down without any reduction:
\[
u_j \coloneqq \text{row}_2(B_j)=\big(\tbinom{r}{0}M_j^{r},\ \tbinom{r}{1}M_j^{r-1}X,\ \dots,\ \tbinom{r}{r}X^r,\ 0\big),
\]
and
\[
v_j \coloneqq \text{row}_3(B_j)-M_j\,\text{row}_2(B_j)
=\big(0,\ \tbinom{r}{0}M_j^{r}X,\ \tbinom{r}{1}M_j^{r-1}X^2,\ \dots,\ \tbinom{r}{r}X^{r+1}\big).
\]
\end{remark}
}
Finally we present the main factoring algorithm. In this algorithm, $N_0$ is a constant that is chosen large enough to ensure that the proof of correctness works for all $N \geq N_0$.

\begin{algorithm}[Factoring $r$-powers]\ \\   % force line break
\label{alg:factorprq}
\textit{Input:} An integer $N=p^{r}q\ge N_0$ with constants $0<c_{1}<c_{2}$ such that $c_{1}N^{1/2}<q<c_{2}N^{1/2}$.

\noindent\textit{Output:}
$p$ and $q$.
\begin{algorithmic}[1]
\If{$r>\frac{\log N}{32\log\log N}$} 
\For{$i=0,\dots,\log N$}
    \If{$i^r|N$} 
        \State Recover $p$ and $q$ and return. 
    \EndIf
\EndFor
\EndIf
\State
Apply Lemma~\ref{lemma:bigOrderprq} with $N$, $r$ and $\delta=N^{1/4r}\log^8 N$.
If any factors of $N$ are found, return. Otherwise, we obtain $\alpha \in \mathbb{Z}_N^*$ with $\ord_{N}(\alpha)>\delta$.
\State\label{step:prqsearch}%
Run Algorithm \ref{alg:searchprq} (the main search)
with the given $N$, {$c_1,c_2$, $m=\left[r^{-1/4}N^{1/4r}\log^{1/2} N\right]$} and $\alpha$.
Return $p$ and $q$.
\end{algorithmic}
\end{algorithm}

\begin{proposition}
Algorithm \ref{alg:factorprq} is correct (for suitable $N_0$),
and it runs in time
\[
O{(r^{1/4} N^{1/4r} \log^{5/2} N)}.
\]
\end{proposition}
\begin{proof}
For correctness, we first handle the case $r > \frac{\log N}{32 \log \log N}$ by enumerating all $p < \log N$. For the case $r < \frac{\log N}{32 \log \log N}$, we can recover $p$ and $q$ via the main search step. Moreover, in this case, we have $2em\sqrt{r}< N^{1/(4r)} \log^8 N$, which satisfies the requirement on the order size needed by the main search.

Regarding the algorithm's complexity: when $r > \frac{\log N}{32 \log \log N}$, the enumeration of $p < \log N$ is negligible. For the case $r < \frac{\log N}{32 \log \log N}$, finding $\alpha$ takes time less than $N^{1/(8r)} \log^4 N$, and the order checking is bounded by $N^{1/(4r)}$. Therefore, the most expensive step is the main search, which takes time 
\[
O\!\left({\sqrt{r}\,m\,\log^2\!N+\frac{N^{1/2r}}{m}\,\log^3\!N}\right)=O({r^{1/4} N^{1/4r} \log^{5/2} N}).
\]

\end{proof}

\begin{remark}
When $r = \Omega\!\left(\dfrac{\log N}{\log \log N}\right)$, we have 
$p = \mathrm{poly}(\log N)$, so $p$ can be found in $\mathrm{poly}(\log N)$ time via exhaustive search. Now consider the case where $r = o\!\left(\dfrac{\log N}{\log \log N}\right)$. 
In this regime, the running time of our algorithm is $O\!\left(r^{1/4} N^{1/4r} \log^{5/2} N\right)$. Compared to Strassen's method, which has complexity $O\left(N^{1/4r} \log^3 N\right)$~\cite[Proposition 2.5]{harvey2021exponent}, our method is strictly better.

\end{remark}

\section{Conclusion}
\label{sec:conclusion}
In this work, we present a novel deterministic integer factorization approach that merges Coppersmith's method with Harvey and Hittmeir's Baby-step Giant-step framework through a specialized rank-$3$ lattice construction. The key insight of our work is utilizing the second vector from LLL-reduced lattice bases—rather than the traditional shortest vector—to generate giant steps that avoid trivial collisions, resulting in more efficient factorization algorithms with improved asymptotic complexity bounds.

Our contributions span multiple factorization problems: For semiprimes with known bit size of $p$ or $q$, we achieve a logarithmic improvement in complexity compared to previous results. For numbers representing sums and differences of powers, our algorithm also provides significant logarithmic improvements. .

We also introduce several technical innovations that may be of independent interest: an improved approach to finding elements of high multiplicative order; a refined generalization of Harvey's deterministic factorization method for identifying $r$-power divisors; and an extension of Coppersmith's method that relaxes determinant constraints while preserving practical utility. 

An open question remains whether our approach could directly improve the general bound established by Harvey and Hittmeir \cite{harvey2022log}. Without additional information about the factors $p$ or $q$, our algorithm currently yields the same complexity bound $O\left(\frac{N^{1/5}\log^{16/5}N}{(\log\log N)^{3/5}}\right)$ as their original work.

\section*{Acknowledgments}
The author thanks the anonymous reviewers for their careful reading and helpful comments that improved both the clarity and quality of this paper. In particular, the author is grateful for the insightful suggestion to consider a BSGS-based acceleration for Algorithm~\ref{alg:findalpha-bsgs}. AI tools were used as typing assistants for grammar and basic editing. Yiming Gao and Honggang Hu were supported by National Cryptographic Science Foundation of China (Grant No. 2025NCSF01001), National Natural Science Foundation of China (Grant No. 62472397), and Quantum Science and Technology-National Science and Technology Major Project (Grant No. 2021ZD0302902). Yansong Feng and Yanbin Pan were supported by the National Natural Science Foundation of China (No. 62372445).

\bibliographystyle{amsalpha}
\bibliography{mybib}

\appendix
\section{Proof of Theorem~\ref{thm:prqfactorsmall}}
\label{sec:appendixProofprq}
\begin{proof}
Without loss of generality, we may assume $s \leq m-1$. Let $p = mx_0 + s$ where $p^r | N$. This implies $p \leq N^{1/r}$ and $|x_0| \leq p/m$. Let $t \equiv m^{-1} \pmod{N}$. We observe that $p | (st + x_0)$.

Define a sequence $\{X_i\}$ where $X_0 = N^{1/r}$, $X_1 = X_0/2, \ldots, X_k = X_0/2^k$ with $k = \lfloor \log N/r \rfloor$. Let $X_i = N^{\beta_i}$ where $\beta_{i+1} = \beta_i - \frac{1}{\log N}$.

For each interval $[X_{i+1}, X_i]$, consider the polynomial $f_r(x) = (st + x)^r$. When $p \in [X_{i+1}, X_i]$, we have:

$$
f_r(x_0) \equiv 0 \pmod{p^r}, \quad |x_0| \leq X_i/m = N^{\beta_i}/m
$$
Applying Lemma \ref{lem:fwh} with $\delta=r,b = p^r$, $b \geq N^\beta = X^r_{i+1}$, $\beta = r\beta_{i+1}$, and $c = X_i/m/N^{\beta^2/r}$, we obtain an upper bound for $c$:
$$
c = \frac{X_i}{mN^{\beta^2/r}} = \frac{N^{\beta_i-r\beta^2_{i+1}}}{m} = \frac{2N^{\beta_{i+1}-r\beta^2_{i+1}}}{m} \leq \frac{2N^{1/4r}}{m}
$$
Thus, for each interval $[X_{i+1}, X_i]$ where $i \in [k-1]$, we can find all $p$ satisfying $p \in [X_{i+1}, X_i]$ and $p^r | N$ in time:
$$
O\left(\lceil c\rceil\frac{\log^{6+3\epsilon} N}{\delta^{1+\epsilon}}\right)=O\left(\left\lceil\frac{ N^{1/4r}} {m}\right\rceil \frac{\log^{6+3\epsilon} N}{r^{1+\epsilon}} \right)
$$
We repeat this process for intervals $[X_k, X_{k-1}], \ldots, [X_2, X_1], [X_1, X_0]$. For $[0, X_k]$, we perform exhaustive search for $p \equiv s \pmod{m}$. With $k = \lfloor \log N/r \rfloor$ intervals, the total time complexity is:
$$
\left\lfloor \frac{\log N}{r}\right\rfloor O\left(\left\lceil\frac{N^{1/4r}}{m}\right\rceil \frac{\log^{6+3\epsilon} N}{r^{1+\epsilon}} \right) = O\left(\left\lceil\frac{N^{1/4r}}{m}\right\rceil \frac{\log^{7+3\epsilon} N}{r^{2+\epsilon}} \right)
$$
Since $X_k = N^{1/r}/2^k = O(1)$, the exhaustive search complexity is $O(\log N)$, which is negligible in the overall complexity.
\end{proof}

\section{Order-finding algorithms for $r$-powers}
\label{sec:appendixFindprq}
\begin{lemma}[~\cite{markus2018babystep}, Theorem 6.1]
\label{lemma:prqFingingOrderLemma}
There exists an algorithm with the following properties:
\begin{itemize}
    \item \textbf{Input:} $N \in \mathbb{N}$, $T \leq N$ and $a \in \mathbb{Z}_N^*$.
    \item \textbf{Output:} If $\operatorname{ord}_N(a) \leq T$, then the output is $\operatorname{ord}_N(a)$; otherwise, the output is '\textit{ord}$_N(a) > T$'.
\end{itemize}
The runtime complexity is bounded by 
$$
O\left(\frac{T^{1/2}}{\sqrt{\log \log T}}  \cdot \log^2 N \right).
$$
\end{lemma}

\begin{algorithm}[Order-finding algorithms]\ \\
\label{alg:prqfindBigAlpha}
\textit{Input:}
\begin{inputoutputlist}
\item $N \in \mathbb{N}$ and $\delta \leq N$. 
\end{inputoutputlist}
\textit{Output:}
\begin{inputoutputlist}
\item Either some $a \in \mathbb{Z}_N^*$ such that $\operatorname{ord}_N(a) > \delta$, or a nontrivial factor of $N$, or ``$N$ is $r$ power free''.
\end{inputoutputlist}
\begin{algorithmic}[1]
\State Set $M_1 = 1$ and $a = 2$
\For{$e = 1, 2, \ldots$} 
    \While{$a \nmid N$ \textbf{and} $a^{M_e} \equiv 1 \pmod{N}$}
        \State $a \gets a + 1$
    \EndWhile
    \If{$a \mid N$}
        \State \Return $a$ as a nontrivial factor of $N$, or if $a = N$, return 'N is prime'.
    \EndIf
    \State Apply Lemma~\ref{lemma:prqFingingOrderLemma} with $T = \delta^{1/2}/\log^2N.$\label{step:findorderfirst}
    \If{$\operatorname{ord}_N(a)$ is not found}
        \State Apply Lemma~\ref{lemma:prqFingingOrderLemma} with $T = \delta.$\label{step:findordersecond}
        \If{$\operatorname{ord}_N(a)$ is not found}
            \State \Return $a$ as an element with $\operatorname{ord}_N(a) > \delta.$
        \EndIf
    \EndIf
    \State Set $m_e = \operatorname{ord}_N(a)$ and compute the prime factorization of $m_e.$ \label{step:factorme}
    \For{each prime $p$ dividing $m_e$}
        \If{$\gcd(N, a^{m_e / p} - 1) \neq 1$}
            \State \Return $\gcd(N, a^{m_e / p} - 1)$ as a nontrivial factor of $N.$
        \EndIf
    \EndFor
    \State Set $M_{e+1} \gets \operatorname{lcm}(M_e, m_e)$
    \If{$M_{e+1} \geq \delta^{1/2}/\log^2 N$} \label{step:testiforderenough}
        \State Apply Theorem~\ref{thm:prqfactorsmall} with $r,s = 1,m = M_{e+1}.$ \label{step:findr-power}
        \State \Return some nontrivial factor of $N$ or ``$N$ is $r$ power free''.
    \EndIf
    \State $a \gets a + 1$
\EndFor
\end{algorithmic}
\end{algorithm}

\begin{theorem}
Algorithm~\ref{alg:prqfindBigAlpha} is correct. Assuming $N^{1/4r}\log^8  N \leq \delta$, the runtime complexity of Algorithm~\ref{alg:prqfindBigAlpha} is bounded by 
\[
O\left(\frac{\delta^{1/2}\log^2 N}{(\log\log \delta)^{1/2}} \right).
\]
\end{theorem}
\begin{proof}
Our algorithm differs from \cite[Algorithm 6.2]{markus2018babystep} in two aspects: we set $T=\delta^{1/2}/\log^2 N$ instead of $T=\delta^{1/3}$ in Steps \ref{step:findorderfirst} and \ref{step:testiforderenough}, and we apply Theorem \ref{thm:prqfactorsmall} for detecting r-powers of $N$ in Step \ref{step:findr-power}. Since Theorem \ref{thm:prqfactorsmall} identifies all primes $p$ satisfying $p \equiv s \pmod{m}$ and $p^r | N$, these modifications preserve the correctness established in the proof of \cite[Theorem 6.3]{markus2018babystep}. We now analyze the time complexity.

Let us first examine the running time of Step \ref{step:findr-power}. When the algorithm reaches this step, we have $M_e \geq \delta^{1 / 2}/\log^2 N$. Given our assumption that $N^{1 / 4r} \log^8 N\leq \delta$, Theorem \ref{thm:prqfactorsmall} bounds the computational cost by:
$$
O\left(\left\lceil\frac{ N^{1/4r}} {M_e}\right\rceil \frac{\log^{7+3\epsilon} N}{r^{2+\epsilon}} \right)=
O\left(\frac{\delta^{1/2}\log^{1+3\epsilon} N}{r^{2+\epsilon}} \right).
$$
This shows that the cost of Step \ref{step:findr-power} is asymptotically negligible.

When the algorithm reaches Step \ref{step:findordersecond}, it terminates within the claimed running time. By Lemma \ref{lemma:prqFingingOrderLemma}, Step \ref{step:findordersecond} requires
$$
O\left(\frac{\delta^{1 / 2}}{\sqrt{\log \log \delta}} \cdot \log ^2 N\right).
$$
If $\operatorname{ord}_N(a)$ is not found, the algorithm terminates. Otherwise, assuming $\operatorname{ord}_N(a) \leq \delta$, we note that $m_e>\delta^{1 / 2}/\log^2 N$ since the algorithm reached Step \ref{step:findordersecond}. This implies $M_e \geq m_e>\delta^{1 / 2}/\log^2 N$, leading to Step \ref{step:findr-power}, whose negligible runtime was analyzed above.

For the $e$-th iteration of the main loop, each while loop execution in Steps 3-4 takes $O(\mathsf{M}(\log N) \log M_e)=O(\log^2N\log\log N)$ operations, with at most $M_e<\delta^{1 / 2}/\log^2 N$ iterations. Thus, the total cost is bounded by 
$$
O\left(\delta^{1 / 2}\log\log N\right).
$$
Step \ref{step:findorderfirst} requires $O\left(\delta^{1 / 4+o(1)}\right)$ bit operations by Lemma \ref{lemma:prqFingingOrderLemma}. If $\operatorname{ord}_N(a)$ is found and since $m_e<\delta^{1 / 2}/\log^2 N$, factoring $m_e$ in Step \ref{step:factorme} via trial division costs $O\left(\delta^{1 / 4+o(1)}\right)$ bit operations. All remaining computations are polynomial-time.

Therefore, one complete main loop iteration requires $O\left(\delta^{1 / 2}\log\log N\right)$ bit operations. Since the value of $M_{e+1}$ is at least twice as large as $M_e$,  it iterates at most $O(\log \delta)$ until $M_e$ reaches $\delta^{1 / 2}/\log^2 N$. The total time until reaching either Step \ref{step:findordersecond} or Step \ref{step:findr-power} is $O\left(\delta^{1 / 2}\log N\log\log N\right)$, which is asymptotically negligible. This completes the proof.
\end{proof}
\end{document}